%% file: MainText.tex
\documentclass{article}
\usepackage{pifont}
\newcommand{\xmark}{\ding{55}}
\usepackage{microtype}
\usepackage{graphicx}
\usepackage{subfigure}
\usepackage{booktabs} 
\usepackage{comment}
\usepackage{multirow}

\usepackage{hyperref}

\usepackage[accepted]{icml2020}

\usepackage[latin9]{inputenc}
\usepackage{amstext}
\usepackage{amsthm}
\usepackage{amssymb}
\usepackage{amsmath}  

\newtheorem{thm}{\protect\theoremname}
\newtheorem{defn}[thm]{\protect\definitionname}
\newtheorem{lem}[thm]{\protect\lemmaname}
\newtheorem*{lem*}{\protect\lemmaname}
\newtheorem{rem}[thm]{\protect\remarkname}

\makeatother

\usepackage[english]{babel}
\providecommand{\definitionname}{Definition}
\providecommand{\lemmaname}{Lemma}
\providecommand{\remarkname}{Remark}
\providecommand{\theoremname}{Theorem}

\icmltitlerunning{Spectral Frank-Wolfe Algorithm: Strict Complementarity and Linear Convergence}

\begin{document}

\twocolumn[
\icmltitle{Spectral Frank-Wolfe Algorithm: Strict Complementarity and Linear Convergence}




\begin{icmlauthorlist}
\icmlauthor{Lijun Ding}{orie}
\icmlauthor{Yingjie Fei}{orie}
\icmlauthor{Qiantong Xu}{cs}
\icmlauthor{Chengrun Yang}{ece}
\end{icmlauthorlist}

\icmlaffiliation{orie}{School of ORIE, Cornell University, Ithaca, NY 14850, USA}
\icmlaffiliation{ece}{School of Electrical and Computer Engineering, Cornell University, Ithaca, NY 14850, USA}
\icmlaffiliation{cs}{Facebook AI Research, Menlo Park, CA 94025, USA}

\icmlcorrespondingauthor{Lijun Ding}{ld446@cornell.edu}

\icmlkeywords{Machine Learning, ICML}

\vskip 0.3in
]



\printAffiliationsAndNotice{}  

\global\long\def\fronorm#1{\|#1\|_{\text{F}}}

\global\long\def\twonorm#1{\|#1\|_{2}}

\global\long\def\opnorm#1{\|#1\|_{\text{op}}}

\global\long\def\nucnorm#1{\|#1\|_{*}}

\global\long\def\infnorm#1{\|#1\|_{\infty}}

\global\long\def\abs#1{\left|#1\right|}

\global\long\def\real{\mathbb{R}}

\global\long\def\integer{\mathbb{Z}}

\global\long\def\inprod#1#2{\langle#1,#2\rangle}

\global\long\def\symMat{\mathbb{S}}

\global\long\def\specsplex{\mathcal{S}_n}

\global\long\def\dm{n}

\global\long\def\cons{m}

\global\long\def\tr{\mathbf{tr}}

\global\long\def\Amap{\mathcal{A}}

\global\long\def\xsol{X_{\star}}

\global\long\def\rsol{r_{\star}}

\global\long\def\zsol{Z_{\star}}

\global\long\def\ssol{s_{\star}}

\global\long\def\face#1{\mathcal{C}_{#1}}

\global\long\def\onevec{\mathbf{1}}

\global\long\def\ddiag{\mbox{ddiag}}

\global\long\def\diag{\mbox{diag}}

\global\long\def\range{\mbox{range}}

\global\long\def\dist{\mbox{dist}}

\global\long\def\nullspace{\mbox{nullspace}}

\global\long\def\rank{\mbox{rank}}

\global\long\def\proj{P}

\global\long\def\relint{\mbox{relint}}

\global\long\def\dualdim{k_\star}

\global\long\def\dspacerep{V_\star}

\global\long\def\EV{\mathbf{EV}}


\newcommand{\blue}[1]{\textcolor{blue}{#1}}

\begin{abstract}
We develop a novel variant of the classical Frank-Wolfe algorithm, which we call spectral Frank-Wolfe, for convex optimization over a spectrahedron. The spectral Frank-Wolfe algorithm has a novel ingredient: it computes a few eigenvectors of the gradient and solves a small-scale SDP in each iteration. Such procedure overcomes slow convergence of the classical Frank-Wolfe algorithm due to ignoring eigenvalue coalescence. We demonstrate that strict complementarity of the optimization problem is key to proving linear convergence of various algorithms, such as the spectral Frank-Wolfe algorithm as well as the projected gradient method and its accelerated version. 
\end{abstract}

\section{Introduction}

We consider solving the following optimization problem with the decision
variable $X\in\real^{\dm\times\dm}$:
\begin{align}
\mbox{minimize}\quad & f(X):=g(\Amap X)+\inprod CX\label{eq:Mainoptimization}\\
\mbox{subject to\ensuremath{\quad}} & \tr(X)=1\quad X\succeq0.\nonumber 
\end{align}

\paragraph{Problem setup.}
The setup of Problem~(\ref{eq:Mainoptimization}) is as follows. We assume $C \in\real^{\dm\times\dm}$ is a symmetric matrix.  The constraint $X\succeq0$ means that $X$ is symmetric and positive semidefinite.
We assume that $\Amap:\symMat^{\dm}\rightarrow\real^{\cons}$ is a linear map
from the set of symmetric matrices $\symMat^{\dm}$ to the $m$-dimensional Euclidean space. We also assume that the
function $g:\real^{\cons}\rightarrow\real$ is differentiable and its
gradient $\nabla g$ is $L_{g}$-Lipschitz continuous. We use $\tr(\cdot)$ to denote the standard
trace operation, the sum of diagonal entries of the input matrix. We denote by $\specsplex$ the feasible region of Problem~(\ref{eq:Mainoptimization}). The set $\specsplex$ is called the spectrahedron, which 
is nonempty and compact. Hence Problem (\ref{eq:Mainoptimization})  always has an optimal solution. 
In this paper, we assume Problem~(\ref{eq:Mainoptimization})
admits a unique optimal solution $\xsol$ with rank $\rsol$ for the sake of simplicity.
The main results, Theorem \ref{thm:MainAlgorithmTheorem} and \ref{thm:MainStructralTheorem} below, can be adapted to the setting where multiple optimal solutions exist; see Section \ref{sec: uniqueness} in the Appendix for a further discussion. It is worth noting that for almost all matrix $C$, the solution of Problem~(\ref{eq:Mainoptimization}) is indeed
unique \citep[Corollary 3.5]{drusvyatskiy2011generic}.

\paragraph{Applications.} The optimization problem covers many low rank matrix recovery problems
including matrix sensing \cite{recht2010guaranteed}, matrix completion \cite{candes2009exact,jaggi2010simple}, phase retrieval
\cite{candes2015phase,yurtsever2017sketchy}, and blind deconvolution \cite{ahmed2013blind}. The constraints $X\succeq0$ and
$\tr(X)=1$ impose low-rankness on the solution. The rank $\rsol$ of optimal solutions in these applications 
is expected to be small comparing to the problem dimension $\dm$. We note that
the following problem:
\begin{equation}
\begin{aligned}\label{eq: tracenormballformulation} 
\mbox{minimize}_{\nucnorm X\leq\alpha}\quad f(X),
\end{aligned} 
\end{equation}
is sometimes a more direct optimization formulation for aforementioned low
rank matrix recovery problems. Since Problem~(\ref{eq: tracenormballformulation}) can be re-formulated
as Problem~(\ref{eq:Mainoptimization})   \cite{jaggi2010simple}, we consider Problem~(\ref{eq:Mainoptimization})
as our main focus of study in this paper.

\paragraph{Background and related works.} A natural but costly algorithm for solving~\eqref{eq:Mainoptimization} is using the projected gradient descent method (PGD) or its accelerated version (APGD) \cite{nesterov2013introductory}. Although 
the iteration complexity of PGD or APGD is considerably low,\footnote{ PGD or APGD achieves an $\epsilon$-approximate solution in $\mathcal{O}(\log(\frac{1}{\epsilon}))$ iterations for strongly convex $f$. APGD achieves an  $\epsilon$-approximate solution $\mathcal{O}(\frac{1}{\sqrt{\epsilon}})$ for general smooth $f$.} 
each of their iteration requires computing a full eigenvalue decomposition of an $\dm \times \dm$ matrix, which scales as 
$\mathcal{O}(n^3)$~\cite{trefethen1997numerical}. The high per-iteration cost prevents their large-scale deployment. 
\begin{algorithm}[tb]
	\caption{Frank-Wolfe with line search}
	\label{alg: Frank_wolfe}
	\begin{algorithmic}
		\STATE {\bfseries Input:} initialization $X_0\in\specsplex$
		\FOR{$t=1,2,\dots,$ }
		\STATE \textbf{Eigenvalue computation:} compute an eigenvector $v$ of $\nabla f(X_t)$ associated with smallest 
		eigenvalue. 
		\STATE \textbf{Line search:} solve $\hat{\eta}=\arg\min_{\eta \in [0,1]}f(\eta X_{t}+(1-\eta)vv^\top )$  and set $X_{t+1} = \hat{\eta} X_{t}+(1-\hat{\eta})vv^\top$.
		\ENDFOR	
	\end{algorithmic}
\end{algorithm}
\begin{algorithm}[tb]
	\caption{Generalized BlockFW (G-BlockFW)}
	\label{alg:generalizedBlockFW}
	\begin{algorithmic}
		\STATE {\bfseries Input:} initialization $X_0\in\specsplex$, a step size $\eta\in[0,1]$, 
		a smooth parameter $\beta$, and an integer $k>0$
		\FOR{$t=1,2,\dots,$ }
		\STATE \textbf{Eigenvalue computation:} compute top $k$ eigenvalues 
		$(\lambda_1,\dots,\lambda_r)$ and their eigenvectors $V=[v_1,\dots,v_k]$ of $X_t-\frac{1}{\eta \beta}\nabla f(X_t)$.
		\STATE \textbf{Eigenvalue projection:} project $(\lambda_1,\dots,\lambda_k)$ to the $k$-dimensional 
		probability simplex $\{x\in \real^k\mid \sum_{i=1}^k x_i =1,x_i\geq 0 \}$, and get the projected point $\Lambda$. 
		\STATE \textbf{Forming a new iterates:} set $X_{t+1}= (1-\eta)X_t+\eta V\diag(\Lambda)V^\top$. 
		\ENDFOR	
	\end{algorithmic}
\end{algorithm}
Hence, projection-free methods are sought, such as the Frank-Wolfe method (FW) \cite{frank1956algorithm,jaggi2013revisiting} presented in Algorithm \ref{alg: Frank_wolfe}. In the spectrahedron setting, each step only 
requires computing \emph{one} eigenvector of the gradient of $f$, which can be efficiently done using the Lanczos method \cite{kuczynski1992estimating} by taking advantage of the structure of $\nabla f(X)=\Amap^*(\nabla g)(\Amap X) +C$ as well as the sparsity of $\Amap$ and $C$.  FW converges to an $\epsilon$-approximate solution\footnote{A matrix $X$ is $\epsilon$-approximate solution to Problem \eqref{eq:Mainoptimization} if $X$ is feasible and $f(X)-f(\xsol)\leq \epsilon$.}
within $\mathcal{O}(\frac{1}{\epsilon})$ many iterations. However, the iteration complexity $\mathcal{O}(\frac{1}{\epsilon})$ is tight as shown in \cite{garber2016faster} even if $f$ is strongly convex and no structural assumption is posed on the solution of \eqref{eq:Mainoptimization}. 
Considerable recent research effort  \cite{garber2016faster,fd2017extended,allen2017linear,garber2019linear} has focused on incorporating the low-rankness
of solution $\xsol$. Of particular relevance to our work are \citet{garber2019linear} and \citet{allen2017linear}:
\begin{itemize}
    \vspace{-0.5cm}
    \setlength\itemsep{0em}
	\item \citet{garber2019linear} shows that Algorithm~\ref{alg: Frank_wolfe}  converges linearly 
	given that the solution is rank \emph{one}, and an eigengap assumption on the gradient $\nabla f(\xsol)$ at the optimal solution is satisfied. We note that the rank-one assumption is crucial for the linear convergence of Algorithm~\ref{alg: Frank_wolfe} to hold. As we will demonstrate in Section \ref{sec: numerics}, if the solution is not rank one, Algorithm~\ref{alg: Frank_wolfe} gets stagnant and behaves in the worst case as  $\mathcal{O}(\frac{1}{\epsilon})$.
	
	  \item  \citet{allen2017linear} proposes an algorithm called BlockFW, which is re-formulated as Algorithm \ref{alg:generalizedBlockFW} for our setting and renamed as  generalized BlockFW(G-BlockFW) \footnote{We note that BlockFW is \emph{not} designed for \eqref{eq:Mainoptimization}, but rather for \eqref{eq: tracenormballformulation}. Since \eqref{eq: tracenormballformulation} covers~\eqref{eq:Mainoptimization}, we renamed the algorithm as G-BlockFW.}. It computes only $k$ eigenvectors in each step, and converges linearly so long as $k\geq \rsol=\rank(\xsol)$ and $f$ is strongly convex.  However, the method relies critically on the assumption $k\geq \rsol$: no convergence guarantees can be made if this assumption fails. Indeed, we will demonstrate in Section \ref{sec: numerics} that if $k<\rsol$, G-BlockFW  gets stuck at moderate accuracy and cannot make further progress.\footnote{\citet{allen2017linear} gives an adaptive $k$ selection procedure which works well in their experiments, but there is no theoretical guarantee for the procedure.} Moreover, the method needs 
	to store iterates explicitly to compute the eigenvectors. This not only incurs an extra $\mathcal{O}(n^2)$ space complexity, but also increases the burden of 
	computing eigenvectors as the iterates themselves have no structure to be exploited for fast eigenvector computation.\footnote{Actually \citet{allen2017linear}  provides a method to avoid the extra space and time costs. However, the method requires knowledge of the strong convexity parameter, which is unavailable in all experiments they perform.}
\end{itemize}
In summary, previous methods converge linearly only when the optimal solution is rank one, or the number of eigenvectors computed in each iteration is no smaller than the rank of the optimal solution.

\begin{table*}\label{table: comparison}
\centering
\begin{tabular}{cccc}
\toprule 
    \multirow{2}{*}{Algorithm} & \multicolumn{3}{c}{Convergence Rate} \\
    \cmidrule(lr){2-4}
	& Worst  & Linear  & Condition  \\
\midrule
FW (Alg.~\ref{alg: Frank_wolfe}) & $\frac{8L_f}{t}$ & $(1-\frac{\delta}{12L_f})^t$ & $\rsol=1$ and strict comp.  \\ 
G-BlockFW (Alg.~\ref{alg:generalizedBlockFW}) & \xmark  & $(1-\frac{\gamma}{2L_f})^t$ & $k\geq \rsol$ and QG \\
SpecFW (Alg.~\ref{alg:spectral_frank_wolfe}) & $\frac{8L_f}{t}$ & $(1-\frac{\min(\delta,\gamma)}{12L_f})^t$ & $k\geq \rsol $, QG, and strict comp. \\
\bottomrule
\end{tabular} 
\caption{Comparision of FW, G-BlockFW and SpecFW. Here, we assume $f$ has gradients $\nabla f$ that are $L_f$-Lipschitz. The optimal
solution rank is $\rsol=\rank(\xsol)$. We let $t$ be the number of iterations.
Convergence rates are measured by $f(X_t)-f(\xsol)$. We set $\delta$ to be the difference
between the smallest eigenvalue and the $(\rsol+1)$th-smallest eigenvalue of 
$\nabla f(\xsol)$, that is, $\delta =\lambda_{n-\rsol}(\nabla f(\xsol))- \lambda_{n}(\nabla f(\xsol))$.  "Strict comp." means strict complementarity (Definition \ref{def:(Strict-Complementarity)-Suppose}). "QG" means quadratic growth with parameter 
$\gamma$ (Definition \ref{def: growthCondition}). Both FW and SpecFW have 
 burn-in phases which are bounded by $\frac{72L_f^3}{(\min\{\gamma,\delta\})^3}$. Here, the burn-in phase 
is the number of iterations in which the method converges with standard
rate $L_f/t$, before shifting to the faster rate (if linear convergence condition is satisfied).
The convergence rate of G-BlockFW can be found in Lemma \ref{lem: linearConvergenceOfBlockFW} in 
Section \ref{sec: lemmaForSection4} of the Appendix.}
\end{table*} 

\paragraph{Our contributions.} The contribution of this work is two-fold. On the problem structure side:
\begin{itemize}
    \vspace{-0.3cm}
    \setlength\itemsep{0em}
	\item We show that the eigengap assumption in \cite{garber2019linear} is equivalent to 
	the strict complementarity condition, a well-known regularity condition of semidefinite programming~\cite{alizadeh1997complementarity};
	see Section~\ref{sec: StrictComplementarityAndGrowth} for more detail.
	\item Based on the eigengap condition, or the equivalent strict complementarity  condition, we show that 
	Problem \eqref{eq:Mainoptimization} satisfies the quadratic growth property (Definition~\ref{def: growthCondition} below)  when 
	the outer function $g$ is strongly convex over the feasible region $\specsplex$ of Problem \eqref{eq:Mainoptimization}, which is true for all the application being considered. This governs the linear convergence of many first order methods such as PGD, APGD, and our method, Spectral Frank Wolfe. 
\end{itemize}
On the algorithm side, we propose a new algorithm called Spectral Frank-Wolfe (SpecFW) in Section \ref{sec: SpectralFrank-WolfeAndItsConvergenceRate}, which has the following 
properties:
\begin{itemize}
    \setlength\itemsep{0.2pt}
	\item In each of its iteration, it computes $k$ eigenvectors using \emph{only} the current gradient information. 
	\item In each of its iteration, it solves a small-scale sub-problem efficiently by APGD for small $k$.
	\item It always converges at the rate $\mathcal{O}(\frac{1}{\epsilon})$ no matter what choice of $k$ is. 
	\item It converges linearly when $k\geq \rsol$, and the strict complementarity and  
	quadratic growth condition are satisfied. In particular, we do \emph{not} require $f$ to be strongly convex or the rank $\rsol$ to be $1$.
	\item It can easily incorporate the matrix sketching idea from \citet{tropp2017practical} and achieves the so-called storage 
	optimality discussed in \citet{yurtsever2017sketchy}. The sketching procedure obviates the need for storing the full decision matrix $X$ throughout iterations, thereby saving $\mathcal{O}(n^2)$ space.\footnote{Interested readers can find the procedure in Section \ref{sec: SpecFWmatrixSketching} in the Appendix.
		We note the matrix sketching idea cannot be combined with G-BlockFW easily to avoid storing $X$, as G-BlockFW uses a sum of the current iterate and current gradient to compute the eigenvectors, which destroys the fast matrix-vector product property of the gradient.}  
\end{itemize}

\paragraph{Organization.} The rest of the paper is organized as follows. 
In Section \ref{sec: StrictComplementarityAndGrowth}, we explain 
the concept of strict complementarity and the  classical Frank-Wolfe algorithm, and how they motivate our Spectral Frank-Wolfe. In Section \ref{sec: SpectralFrank-WolfeAndItsConvergenceRate}, we present the Spectral Frank-Wolfe and its convergence guarantees. In Section \ref{sec: QuadraticGrowthAndLinearConvergence}, we show that the strict complementarity enforces the quadratic growth condition whenever $g$ is strongly convex on $\specsplex$. Finally, we demonstrate numerically the effectiveness of the Spectral Frank-Wolfe in Section~\ref{sec: numerics}. 

\paragraph{Notation.} For a symmetric matrix $A\in\protect\symMat^{\protect\dm}$,
we denote its $i$-th largest eigenvalue as $\lambda_{i}(A)$. The operator two norm, nuclear norm,
and Frobenius norm are denoted as $\opnorm{A}$, $\nucnorm{A}$, and $\fronorm{A}$, respectively. The inner product $\inprod{\cdot}{\cdot}$ on 
symmetric matrices is the standard trace inner product. We also equip $\mathbb{R}^m$ with the dot product.
For a linear map $\mathcal{B:\protect\symMat}^{d}\rightarrow\protect\real^{l}$, the
adjoint map of $\mathcal{B}$ is denoted as $\mathcal{B}^{*}$. 
We also define its largest and smallest singular values as $\protect\opnorm{\mathcal{B}}=\sigma_{\max}(\mathcal{B})=\max_{\protect\fronorm A=1}\protect\twonorm{\mathcal{B}(A)}$
and $\sigma_{\min}(\mathcal{B})=\min_{\protect\fronorm A=1}\protect\twonorm{\mathcal{B}(A)}$.
Given a matrix $V\in\protect\real^{d\times r}$, we denote the restriction of $\mathcal{B}$ to $V$ as  $\mathcal{B}_{V}:\protect\symMat^{r}\rightarrow\protect\real^{l}$
by $\mathcal{B}_{V}(S)=\mathcal{B}(VSV^{\top})$ for any $S\in\protect\symMat^{r}$. 

\section{Motivating SpecFW from complementarity and Frank-Wolfe}\label{sec: StrictComplementarityAndGrowth}
In this section, we explain the motivations of the  spectral 
Frank-Wolfe from strict complementarity and its relationship with 
the classical Frank-Wolfe.
\subsection{Observation from complementarity}
Let first introduce the KKT condition to see what complementarity means. 
\paragraph{KKT condition.} By Slater's condition for \eqref{eq:Mainoptimization} and the fact that the feasible region $\specsplex$ is compact, the following KKT condition of (\ref{eq:Mainoptimization}) always holds: there is some
dual optimal solution $\zsol\succeq 0$ and $\ssol\in\real$ such that\footnote{If there are multiple primal optimal solutions, then the KKT condition holds for any one of them.}
\begin{align}
\nabla f(\xsol)-\zsol-\ssol I & =0,\quad\text{(First Order Condition)}\label{eq:KKT}\\
\inprod{\zsol}{\xsol} & =0,\quad\text{(Complementarity)}\nonumber \\
\tr(\xsol) & =1,\quad\text{(Linear Constraint Feasibility)}\nonumber \\
\zsol,\xsol & \succeq0.\quad\text{ (PSD Feasibility)}\nonumber 
\end{align}
Here $I$ is the identity matrix in $\symMat^{\dm}$. 
We prove in Lemma \ref{lem: uniquenessOfTheDual} in the Appendix
that the dual 
solution $(\zsol,\ssol)$ is actually unique. 
   
\paragraph{Complementarity: extract $\xsol$ from $\zsol$.} we first note that using $\zsol,\xsol\succeq0$ and complementarity  $\inprod{\zsol}{\xsol}=0$,
we have $\zsol\xsol=0$. This equality implies that 
\begin{align}\label{eq: rangexsolNullspaceZsol}
\range(\xsol)\subset\nullspace(\zsol),
\end{align}
and 
\begin{align}\label{eq: inequalityrankxsolnullspacedimZ}
\rsol=\rank(\xsol)\leq\dim(\nullspace(\zsol))=:\dualdim.
\end{align}

Hence, if we can compute a matrix $\dspacerep \in \real^{\dm \times \dualdim}$ with orthonormal 
columns that span the null space of $\zsol$, and solve for 
\begin{equation}
\begin{aligned} \label{eq: reducedProblem}
S_\star = \arg\min_{S\in \mathcal{S}_{\dualdim}} f(\dspacerep S \dspacerep^\top),
\end{aligned}
\end{equation}
then we get the primal optimal solution $\xsol =\dspacerep S_\star \dspacerep^\top$.

We note that it is 
necessary to optimize over the $\dualdim$-spectrahedron $\mathcal{S}_{\dualdim}$ instead of just a 
$\dualdim$-dimensional probability simplex, as $\dspacerep$ may not be the eigenvectors of $\xsol$ for $\dualdim>1$. 
Problem 
\eqref{eq: reducedProblem} can be solved by APGD rapidly so long as $\dualdim$, the \emph{size}
of $S$, is small.

This naturally leads to the following questions: 
\begin{enumerate}
	\item Problem \eqref{eq: reducedProblem} is easy to solve only if $\dualdim$ is small; yet for now we only have 
	$\dualdim\geq \rsol$. With $\rsol$ expected to be small, can we hope for $\dualdim =\rsol$ to hold, 
	so that $\dualdim$ is small as well?
	\item Suppose we have $\dualdim=\rsol$, can we compute $\dspacerep$ exactly or approximate it well enough?
\end{enumerate}

We answer the first question in the next section by defining strict complementarity 
and establishing its equivalence to an eigengap condition on $\nabla f(\xsol)$. 
To answer the second question, we draw relationship between the first order condition in \eqref{eq:KKT} 
and the classical Frank-Wolfe algorithm in Section \ref{sec: fwmotivation}. 

\subsection{Strict complementarity}\label{sec: strictcomplementarity}
We answer why we expect $\rsol =\dualdim$ in this section. 
Using the rank-nullity theorem, we see that the equation $\rsol=\rank(\xsol)\leq\dim(\nullspace(\zsol))$
is equivalent to 
\[
\rank(\xsol)+\rank(\zsol)\leq\dm.
\]
Strict complementarity \cite{alizadeh1997complementarity} 
assumes that we have equality instead of
inequality.
\begin{defn}
	\label{def:(Strict-Complementarity)-Suppose}(Strict Complementarity)
	Let $\xsol,\zsol$ and $\ssol$ satisfy the KKT condition (\ref{eq:KKT}).
	We say that Problem (\ref{eq:Mainoptimization}) (or the pair $(\xsol,\zsol)$) satisfies strict complementarity if
	\[
	\rank(\xsol)+\rank(\zsol)=\dm.
	\]
\end{defn}
It is immediately clear that using the rank-nullity theorem again, we see that strict complementarity 
is equivalent to 
\[
\rsol =\dualdim, 
\]
which is what we desire. 
By \eqref{eq: rangexsolNullspaceZsol} and given that the solution rank is $\rsol$, 
strict complementarity is equivalent to 
\begin{equation}
\lambda_{n-\rsol}(\zsol)>0.\label{eq:equivalenceOfStrictComplementarity}
\end{equation}
Equation~\eqref{eq: rangexsolNullspaceZsol} also implies that we always have for all $i=1,\dots,\rsol$,
\begin{align}
\lambda_{n-\rsol+i}( & \zsol)=0.\label{eq:zsolsmalleigenvaluearezero}
\end{align}
\paragraph*{Relation with the eigengap assumption.}

In \citet{garber2019convergence,garber2019linear}, the author proposed an eigengap condition:
\[
\lambda_{n-\rsol}(\nabla f(\xsol))-\lambda_{n}(\nabla f(\xsol))>0.
\]
This is in fact  equivalent to strict complementarity: since $\nabla f(\xsol)=\zsol+\ssol I$,
we have 
\begin{align*}
&\lambda_{n-\rsol}(\nabla f(\xsol))-\lambda_{n}(\nabla f(\xsol))\\
 =& \lambda_{n-\rsol}(\zsol+\ssol I)-\lambda_{n}(\zsol+\ssol I)\\
 = &\lambda_{n-\rsol}(\zsol)+\ssol-\lambda_{n}(\zsol)-\ssol\\
 = & \lambda_{n-\rsol}(\zsol),
\end{align*}
where the last step is due to (\ref{eq:zsolsmalleigenvaluearezero}). Using 
(\ref{eq:equivalenceOfStrictComplementarity}), we deduce the equivalence.

\paragraph*{Why strict complementarity should hold.}

Strict complementarity as shown in \citet{drusvyatskiy2011generic} holds for almost all $C$ (see 
Lemma \ref{lem: genericstrictcomplementarity} for a more detailed derivation).
We will also verify this assumption numerically in our experiments in Section~\ref{sec: numerics}.
Moreover, as demonstrated in \citet[Lemmas 2 and  10]{garber2019linear}, such assumption should hold if
we expect the solution rank $\rsol$ to be stable under small
perturbations. 

\subsection{FW and approximation of $\nullspace(\zsol)$}\label{sec: fwmotivation}
We have just argued why we expect $\rsol =\dualdim$ should hold for Problem \eqref{eq:Mainoptimization}. In this section, 
we draw relation of FW and approximation of $\nullspace(\zsol)$. 

Denote by $\EV_r(A)$ the eigenspace of the smallest $r$ eigenvalues of a matrix $A\in \symMat^{\dm}$. 
In view of the first order condition \eqref{eq:KKT}, we have  
\begin{equation}
\begin{aligned}\label{eq: gradientAndDual}
\EV_{\dualdim}(\nabla f(\xsol)) = \nullspace(\zsol). 
\end{aligned} 
\end{equation} 
Hence $ \nullspace(\zsol)$ can be identified using the gradient of $f$ at $\xsol$.  

Note that FW indeed uses the eigenvector corresponding to the smallest eigenvalue of $\nabla f(X_t)$ in each of 
its  iteration, and therefore it tries to approximate $\EV_{\dualdim}(\nabla f(\xsol))$. This is 
the main intuition that linear convergence of FW can be established  when $\rsol=1$  as in \citet{garber2019linear}. It also reveals that FW fails to converge in a linear rate for $\dualdim>1$, as approximation using 
one eigenvector is not enough for a $\dualdim$-dimensional space. Also, 
from \eqref{eq:zsolsmalleigenvaluearezero} and the first order condition in 
the KKT condition, we see the smallest 
$\dualdim$ eigenvalues of the gradient coalesce, and hence it 
is important to compute the $\dualdim$-dimensional space 
to attain 
better numerical stability and accuracy.  
Hence, to overcome this issue, we need to  compute at least $\dualdim$ eigenvectors  and solve a sub-problem like \eqref{eq: reducedProblem} in each iteration. 

The above discussion motivates our algorithm, the Spectral Frank-Wolfe (Algorithm \ref{alg:spectral_frank_wolfe}),
described in the next section. 

\section{Spectral Frank-Wolfe and its Convergence guarantees}\label{sec: SpectralFrank-WolfeAndItsConvergenceRate}
In this section, we describe the Spectral Frank-Wolfe algorithm and its theoretical 
guarantees.
\subsection{The Spectral Frank-Wolfe algorithm}
The Spectral Frank-Wolfe algorithm is presented in Algorithm~\ref{alg:spectral_frank_wolfe}. We highlight its key mechanism as follows.
\begin{algorithm}[tb]
	\caption{Spectral Frank-Wolfe}
	\label{alg:spectral_frank_wolfe}
	\begin{algorithmic}
		\STATE {\bfseries Input:} initialization $X_0\in\specsplex$, an integer $k>0$
		\FOR{$t=1,2,\dots,$ }
		\STATE \textbf{Eigenvalue computation:} compute the $k$ eigenvectors, $v_1,\dots,v_k$ of $\nabla f(X_t)$ associated with the $k$ smallest 
		eigenvalues, and form the matrix $V = [v_1,\dots,v_k]\in \real^{\dm \times k}$.
		\STATE \textbf{Solving a small-scale SDP:} solve $\min_{\eta+\tr(S)=1,S\succeq 0,\eta \geq 0}f(\eta X_{t}+VSV^{\top})$ and get an optimal solution $(\hat{S},\hat{\eta})$. 
		\STATE \textbf{Forming a new iterate:} set $X_{t+1} = \hat{\eta} X_{t}+V\hat{S}V^{\top}$.
		\ENDFOR	
	\end{algorithmic}
\end{algorithm}

\paragraph{Solving a small-scale SDP.} 
The small-scale semidefinite programming (SDP) 
\begin{equation}
\begin{aligned}\label{eq: smallksubproblem}
\min_{\eta+\tr(S)=1,S\succeq 0,\eta \geq 0}f(\eta X_{t}+VSV^{\top}).
\end{aligned}
\end{equation}
can be solved easily using APGD since  projection to 
the set $\{(\eta,S)\mid \eta+\tr(S)=1,S\succeq 0, \eta\geq 0\}$ only requires an eigenvalue decomposition of a symmetric matrix of size $k$ and a projection to the $(k+1)$-dimensional probability simplex. The correctness of the procedure for projection can be verified using arguments in  \citet[Lemma 3.1]{allen2017linear},
and \citet[Lemma 6]{garber2019convergence}.
We note that when evaluating gradient is very expensive, instead of minimizing $f(\eta X_{t}+VSV^{\top})$, one can also minimize an upper bound 
of it (and the guarantees in the next section continue to hold). This is discussed in Section \ref{sec: upperboundf} in the Appendix. 

\paragraph{Averaging with current $X_t$.} In addition to the eigenvectors from the current gradient, we also 
utilize the information of previous iterates when solving the small-scale SDP \eqref{eq: smallksubproblem}.
This follows the same spirit as the classical Frank-Wolfe, which performs a line search over the current iterate and the new 
atom $vv^\top$. This averaging scheme stabilizes the algorithm and facilitates the $\mathcal{O}(\frac{1}{\epsilon})$ 
convergence rate. 

\paragraph{The choice of $k$.} From the proof of the convergence in the next section, it can be observed that 
so long as $k\geq \dualdim$, Algorithm~\ref{alg:spectral_frank_wolfe}  converges linearly. Of course, one may not know $\dualdim$ in 
advance. In this case, $k$ may be taken as the largest value subject to the user's computational budget or the 
largest rank of the solution the user can afford in terms of storage. An adaptive strategy 
may also be employed based on the progress of objetive value decay as in \citet[Section 6.2]{allen2017linear}. We do not further the discussion of this issue due to the space limit. 

\subsection{Theoretical  guarantees}

To state our result, we first define the notion of quadratic growth. 

\begin{defn}[Quadratic Growth (QG)]\label{def: growthCondition}
	We say that the optimization problem \eqref{eq:Mainoptimization} satisfies quadratic growth with parameter $\gamma>0$, if for every feasible $X\in \specsplex$ there holds
	\[
		f(X)-f(\xsol)\geq\gamma \fronorm{X-\xsol}^2.
		\]
\end{defn}

The quadratic growth condition is necessary for linear convergence of 
gradient descent type methods as shown in ~\citet[Theorem 13]{necoara2019linear}. 
Hence we should expect it to hold if we are to show linear convergence of Frank-Wolfe methods.
The condition automatically holds for strongly convex $f$, and more broadly,
it is satisfied for almost all $C$ so long as 
$g$ is semi-algebraic, as shown in~\citet[Corollary 4.8]{drusvyatskiy2016generic}. 
In Section \ref{sec: QuadraticGrowthAndLinearConvergence}, we show that
strict complementarity and strong convexity of the outer function $g$ (but not $f$)
implies quadratic growth, as well as an explicit formula of $\gamma$ in terms of the
solution $\xsol$, the map $\Amap$, and smoothness and strong convexity parameters
of $g$. 
 
We now state the theoretical guarantees for our Algorithm~\ref{alg:spectral_frank_wolfe}. 
\begin{thm}
	\label{thm:MainAlgorithmTheorem}Suppose strict complementarity holds
	for Problem (\ref{eq:Mainoptimization}), the optimal solution $\xsol$
	is unique with rank $\rsol$, the function $g$ has $L_{g}$-Lipschitz continuous
	gradients, Problem (\ref{eq:Mainoptimization}) satisfies quadratic growth
	with parameter $\gamma$, and the choice of $k$ satisfies $k\geq\rsol=\dualdim$. Define $h_{t}=f(X_{t})-f(\xsol)$
	for each $t$, and $\beta=\opnorm{\Amap}^{2}L_{g}$. Then for all
	$t$,
	we have 
	\begin{equation}
	f(X_{t})-f(\xsol)\leq\frac{8\beta}{t}.\label{eq:firstpartofthm3}
	\end{equation}
	For all $t\geq T_{0}=\frac{72\beta^{3}}{\gamma\lambda_{n-\rsol}^{2}(\zsol)}$, we have 
	\begin{equation}
	h_{t+1}\leq\left(1-\min\left\{ \frac{\gamma}{4\beta},\frac{\lambda_{n-\rsol}(\zsol)}{12\beta}\right\} \right)h_{t}.\label{eq:secondparthm3}
	\end{equation}
\end{thm}

\paragraph{Discussion on the assumptions.} As discussed before, these assumptions 
are expected to be necessary for linear convergence and robustness of the rank under small perturbations. The assumption of the unique optimal solution is only for the purpose of clear presentation. 

\paragraph{Preparation of the proof.} Let us first give the definition of the $r$-th spectral
set. 
\begin{defn}
	For each $X\in\symMat^{\dm},$ let $V_{X}\in\real^{\dm\times k}$
	having orthonormal eigenvectors as columns corresponding to the smallest
	$k$ eigenvalues of $X$. Define the spectral $k$-th set $\face k(X)$
	of $X$ as 
	\begin{align*}
	\face k(X) :=
	&\left\{ V_{X}SV_{X}^{\top}\in\symMat^{\dm}\mid S\in\mathcal{S}_k\right\} .
	\end{align*}
\end{defn}
We next present the following important lemma which is proved
in Section \ref{sec: AppendixSpectralFrank-WolfeAndItsConvergenceRate} in the Appendix.
\begin{lem}
	\label{lem:Let--and}Given $Y\in\symMat^{\dm}$ which satisfies $\lambda_{\dm-r}(Y)-\lambda_{\dm-r+1}(Y)\geq\delta$
	for some $\delta>0$, then for any $X\in\symMat^{\dm}$, $X\succeq0$,
	and $\tr(X)=1$, there is some $W\in\face r(Y)$ such that 
	\[
	\inprod{X-W}Y\geq\frac{\delta}{2}\fronorm{X-W}^{2}.
	\]
\end{lem}
We are now ready to start the proof.
\begin{proof}[Proof of Theorem \ref{thm:MainAlgorithmTheorem}] Using the Lipschitz
	smoothness of $f$, we have for any $t\geq 1$,
	$\eta\in[0,1]$, and any $W\in\face{\rsol}(\nabla f(X_{t}))$:
	\begin{equation} \label{eq: QuadraticApproximationDueTosmoothness}
	\begin{aligned}
	f(X_{t+1})  \leq &f(X_{t})+(1-\eta)\inprod{W-X_{t}}{\nabla f(X_{t})}
	\\ &  +\frac{(1-\eta)^{2}\beta}{2}\fronorm{W-X_{t}}^{2}.
	\end{aligned}
	\end{equation}
	Now choose $W=v_{n}v_{n}^{\top}$where $v_{n}$ is the eigenvector
	of $\nabla f(X_{t})$ with the smallest eigenvalue, we can then perform
	the analysis as normal Frank-Wolfe as is done in \cite{jaggi2013revisiting}
	to reach the first part  of the theorem, the inequality  (\ref{eq:firstpartofthm3}). 
	
	For the second part, we first note that by the discussion after the
	Definition \ref{def:(Strict-Complementarity)-Suppose} of strict
	complementarity, we have $\lambda_{n-\rsol}\left(\nabla f(\xsol)\right)-\lambda_{n-\rsol+1}(\nabla f(\xsol))=\lambda_{n-\rsol}(\zsol)$,
	and $\lambda_{n-\rsol+1}(\nabla f(\xsol))=\dots=\lambda_{n}(\nabla f(\xsol))$. 
	
	Using Lipschitz continuous gradient of $f$ in step $(a)$, the quadratic growth of $f$ in step $(b)$, and
	the choice of $T_0$ in step $(c)$,  we find that for all $t\ge T_{0}$, 
	\begin{align}
	\fronorm{\nabla f(X_{t})-\nabla f(\xsol)} &\overset{(a)}{\leq}\beta\fronorm{X_{t}-\xsol \nonumber}\\
&\overset{(b)}{\leq}\beta \left(\frac{f(X_{t})-f(\xsol)}{\gamma}\right)^{\frac{1}{2}} \nonumber\\
	&\overset{(c)}{\leq}\frac{1}{3}\lambda_{n-\rsol}(\zsol). \label{eq: smoothnessgradientnearoptimalgradient}
	\end{align}
	
	Using the inequality \eqref{eq: smoothnessgradientnearoptimalgradient} and Weyl's inequality, we find that 
	\begin{equation} \label{eq:gapdifference}\
	\begin{aligned}
	&\lambda_{n-\rsol}(\nabla f(X_{t}))-\lambda_{n-\rsol+1}(\nabla f(X_t)) \\
	=& \underbrace{\lambda_{n-\rsol}\left(\nabla f(\xsol)\right)-\lambda_{n-\rsol+1}(\nabla f(\xsol))}_{=\lambda_{n-\rsol}(\zsol)}\\
	& +\underbrace{\left(\lambda_{n-\rsol}(\nabla f(X_{t}))-\lambda_{n-\rsol}\left(\nabla f(\xsol)\right)\right)}_{\geq-\frac{1}{3}\lambda_{\dm-\rsol}(\zsol)}\nonumber \\
	& +\underbrace{\left(\lambda_{n-\rsol+1}\left(\nabla f(\xsol)\right)-\lambda_{n-\rsol+1}(\nabla f(X_{t}))\right)}_{\geq-\frac{1}{3}\lambda_{n-\rsol}(\zsol)}\\
	& \geq\frac{1}{3}\lambda_{n-\rsol}(\zsol).\nonumber 
	\end{aligned}
	\end{equation}
	Now we subtract the inequality \eqref{eq: QuadraticApproximationDueTosmoothness} both sides by $f(\xsol)$, and denote
	$h_{t}=f(X_{t})-f(\xsol)$ for each $t$, we reach 
	\begin{equation}\label{eq:htht+1intermediatequantity}
	\begin{aligned} 
	h_{t+1}\leq & h_{t}+(1-\eta)\underbrace{\inprod{W-X_{t}}{\nabla f(X_{t})}}_{R_1}\\
	            & +\frac{(1-\eta)^{2}\beta}{2}\underbrace{\fronorm{W-X_{t}}}_{R_2}.
    \end{aligned} 
	\end{equation}
	Using Lemma \ref{lem:Let--and} and the inequality \eqref{eq:gapdifference},
	we can choose $W\in\face{\rsol}(\nabla f(X_{t}))$ such that 
	\begin{align}
	\inprod{W-\xsol}{\nabla f(X_{t})} & \leq-\frac{\lambda_{n-\rsol}(\zsol)}{6}\fronorm{\xsol-W}^{2}.\label{eq:intermediateinequality2forlinearconvergence}
	\end{align}
	Let us now analyze the term $R_1=\inprod{W-X_{t}}{\nabla f(X_{t})}$
	using (\ref{eq:intermediateinequality2forlinearconvergence}) and convexity of $f$:
	\begin{equation*}
	\begin{aligned} \label{eq:inprodw-xtalgorithmtheorem}
	 R_1=   & \inprod{W-X_{t}}{\nabla f(X_{t})}\\
	  = & \inprod{W-\xsol}{\nabla f(X_{t})}+\inprod{\xsol-X_{t}}{\nabla f(X_{t})}\nonumber \\
	\leq&-\frac{\lambda_{n-\rsol}(\zsol)}{6}\fronorm{\xsol-W}^{2}-h_{t}.
	\end{aligned}
    \end{equation*}
	The term $R_2=\fronorm{X_{t}-W}^{2}$ can be bounded by 
	\begin{equation*}
	\begin{aligned} 
	R_2 = \fronorm{X_{t}-W}^{2} &\overset{(a)}{\leq}2\left(\fronorm{X_{t}-\xsol}^{2}+\fronorm{\xsol-W}^{2}\right)\\
	&\overset{(b)}{\leq}\frac{2}{\gamma}h_{t}+2\fronorm{\xsol-W}^{2},\label{eq:xt-walgorithminequality}
	\end{aligned} 
	\end{equation*}
	where we use triangle inequality and the basic inequality $(a+b)^{2}\leq2a^{2}+2b^{2}$
	in step $(a)$, and the quadratic growth condition in step $(b)$. 
	
	Now combining (\ref{eq:htht+1intermediatequantity}), and the bounds of $R_1$ and $R_2$, we reach that there is a
	$W\in\face{\rsol}(\nabla f(X_{t}))$ such that for any $\xi=1-\eta\in[0,1]$, we have 
	\begin{align*}
	h_{t+1} 
	\leq & h_{t}+\xi\left(-\frac{\lambda_{n-\rsol}(\zsol)}{6}\fronorm{\xsol-W}^{2}-h_{t}\right)\\
	+&\frac{\xi^{2}\beta}{2}\left(\frac{2}{\gamma}h_{t}+2\fronorm{\xsol-W}^{2}\right)\\
   =&\left(1-\xi+\frac{\xi^{2}\beta}{\gamma}\right)h_{t}\\
    +&\left(\xi^{2}\beta-\frac{\xi\lambda_{n-\rsol}(\zsol)}{6}\right)\fronorm{\xsol-W}^{2}.
	\end{align*}
	A detailed calculation and choice of $\xi$ in Section \ref{sec: AppendixSpectralFrank-WolfeAndItsConvergenceRate} in Appendix reveals that we can reach the second part of the theorem, the inequality \eqref{eq:secondparthm3}.
\end{proof}

\section{Quadratic Growth and Linear Convergence of Algorithms}\label{sec: QuadraticGrowthAndLinearConvergence}

In this section, we show that when $g$ is $\alpha$-strongly convex \cite{nesterov2013introductory}
and strict complementarity of (\ref{eq:Mainoptimization}) holds, then
we have quadratic growth of Problem (\ref{eq:Mainoptimization}).
We also demonstrate when the dual matrix $\zsol$ has rank $n-1$ then we do not require $g$ to be $\alpha$-strongly convex. An immediate consequence is 
the linear convergence of PGD and APGD \cite{karimi2016linear}, the generalized blockFW\footnote{We show its convergence 
under quadratic growth in Lemma \ref{lem: linearConvergenceOfBlockFW} in Section \ref{sec: lemmaForSection4} in the Appendix.} (Algorithm~\ref{alg:generalizedBlockFW}, and
the spectral Frank-Wolfe (Algorithm~\ref{alg:spectral_frank_wolfe}) as shown in Theorem \ref{thm:MainAlgorithmTheorem}.  
\begin{thm}
	\label{thm:MainStructralTheorem}Suppose strict complementarity of
	(\ref{eq:Mainoptimization}) and one of the following conditions
	hold:
	
	(i) $g$ is $\alpha$-strongly convex, and the solution $\xsol$ is
	unique, or
	
	(ii) the dual matrix $\zsol$ in the KKT condition (\ref{eq:KKT})
	has rank $n-1$,
	
	then Problem~(\ref{eq:Mainoptimization}) satisfies quadratic
	growth. The constant $\gamma$ takes the form of 
	
	(i) $\gamma=\min\left\{ \frac{\lambda_{n-\rsol}(\zsol)}{4+8\frac{\sigma_{\max}^{2}(\tilde{\Amap})}{\sigma_{\min}^{2}(\tilde{\Amap}_{V})}},\frac{\alpha\sigma_{\min}^{2}(\tilde{\Amap}_{V})}{8}\right\} $
	in the first case, where $\tilde{\Amap}(X)=\begin{bmatrix}\tr(X)\\
	\Amap(X)
	\end{bmatrix}$, and 
	
	(ii) $\gamma=\frac{\lambda_{n-\rsol}(\zsol)}{2}$ in the second case. In addition, the uniqueness of $\xsol$ is  implied in the second case. 
\end{thm}

\begin{proof}
	The second case has been verified in \citet[Lemmas 1 and 2]{garber2019convergence}. We 
	provide a self-contained and different proof in Section \ref{sec: lemmaForSection4} in 
	Appendix. 
	
	Now consider the first case. For any feasible $X$ and the optimal
	solution $\xsol,$ we have 
	\begin{equation} \label{eq: optimalitygStrongconvexityInequality}
	\begin{aligned}
            	    &f(X)-f(\xsol)\\
                   =& g(\Amap X)-g(\Amap\xsol)+\inprod C{X-\xsol}\\
\overset{(a)}{\geq}	& \inprod{(\nabla g)(\Amap\xsol)}{\Amap(X-\xsol)}\\
               	   +&\inprod C{X-\xsol}+\frac{\alpha}{2}\twonorm{\Amap X-\Amap\xsol}^{2} \\
\overset{(b)}{=}	& \inprod{\Amap^{*}(\nabla g)(\Amap\xsol)+C}{X-\xsol}\\
	               +&\frac{\alpha}{2}\twonorm{\Amap X-\Amap\xsol}^{2} \\
	\overset{(c)}{=}& \inprod{\zsol+\ssol I}{X-\xsol}+\frac{\alpha}{2}\twonorm{\Amap(X-\xsol)}^{2} \\
\overset{(d)}{=}	& \inprod{\zsol}X+\frac{\alpha}{2}\twonorm{\Amap(X-\xsol)}^{2}\overset{(e)}{\geq}0
	\end{aligned}
    \end{equation}
	Here step $(a)$ is due to the strong convexity of $g$. Step $(b)$
	is because of the definition of $\Amap^*$. For step $(c)$, we uses
	the first order condition of KKT condition (\ref{eq:KKT}) in terms
	of $g$ and $\Amap$: $\Amap^{*}(\nabla g)(\Amap\xsol)+C-\zsol-\ssol I=0.$
	The step $(d)$ is due to the complementarity in KKT condition (\ref{eq:KKT})
	and feasibility of $X$ and $\xsol.$ The last inequality $(e)$ is
	beacause $Z,X\succeq0.$
	
	We claim that a feasible matrix $X\in\symMat^{\dm}$ is optimal if and only
	if $X$ satisfies
	\begin{equation} \label{eq:linearmatrixinequalityoptimalcondition} 
	\begin{aligned}
	\inprod{\zsol}X  =0,\quad 
	\Amap X -\Amap\xsol=0,\\
	\tr(X)  =1,\quad\text{and}\quad 
	X  \succeq0.
	\end{aligned}
    \end{equation}  
	Indeed, if $X$ is optimal, then \eqref{eq: optimalitygStrongconvexityInequality}  
	and feasibility of $X$ implies \eqref{eq:linearmatrixinequalityoptimalcondition}. 
    Conversely, if $X$ satisfies \eqref{eq:linearmatrixinequalityoptimalcondition}, then it 
    satisfies the KKT condition 
    \eqref{eq:KKT} and hence it is optimal because the problem \eqref{eq:Mainoptimization} is convex. 
	Since the optimal solution is unique by assumption, we know the system
	(\ref{eq:linearmatrixinequalityoptimalcondition}) admits a unique
	solution. Using Lemma \ref{lem:Suppose-the-following} in Section \ref{sec: lemmaForSection4} in the Appendix, we have the
	relationship between $(\inprod{\zsol}X,\twonorm{\Amap(X-\xsol)})$
	and the distance to the solution $\fronorm{X-\xsol}$:
	\begin{equation}\label{eq:xxsolzxax-brelation}
	\begin{aligned} 
	\fronorm{X-\xsol}^{2}&\leq\left(4+8\frac{\sigma_{\max}^{2}(\tilde{\Amap})}{\sigma_{\min}^{2}(\tilde{\Amap}_{V})}\right)\frac{\inprod{\zsol}X}{\lambda_{n-\rsol}(\zsol)}\\
	&+\frac{4}{\sigma_{\min}^{2}(\tilde{\Amap}_{V})}\twonorm{\Amap(X)-b}^{2}.
	\end{aligned} 
	\end{equation}
	Combining  \eqref{eq: optimalitygStrongconvexityInequality} and (\ref{eq:xxsolzxax-brelation}),
	we see that 
	\[
	f(X)-f(\xsol)\geq\gamma\fronorm{X-\xsol}^{2}
	\]
	for $\gamma=\min\left\{ \frac{\lambda_{n-\rsol}(\zsol)}{4+8\frac{\sigma_{\max}^{2}(\tilde{\Amap})}{\sigma_{\min}^{2}(\tilde{\Amap_{V})}}},\frac{\alpha\sigma_{\min}^{2}(\tilde{\Amap}_{V})}{8}\right\} $.
\end{proof}

\section{Numerics}\label{sec: numerics}
In this section, we verify numerically a few of our claims in the paper, 
and show the advantages of the Spectral Frank-Wolfe algorithm when strict complementarity 
is satisfied and the solution rank is larger than $1$. We focus on the 
quadratic sensing problem \cite{chen2015exact}. Given a random matrix $U_{\natural}\in \real^{\dm\times r_\natural}$ with 
$r_\natural =3$ and Frobenius norm $\fronorm{U_\natural }^2 =1$,
we generate Gaussian vectors $a_i\in \real^{\dm\times 1},i=1,\dots,m$ and construct  
quadratic measurement vectors $y_0(i) = \fronorm{U_\natural ^\top a_i}^2,i=1,\dots,m$. We then add noise 
$\mathtt{n}=c \twonorm{y_0}v$, where $c$ is the inverse signal-to-noise ratio and $v$ is a random unit vector. Our 
observation is given by $y= y_0 + \mathtt{n}$ and we aim to recover $U_\natural U_\natural ^\top$ from $y$. To this end, we solve the following optimization problem:
\begin{equation}\label{eq: quadraticSensing}
\begin{aligned}
\mbox{minimize}\quad & f(X):=\frac{1}{2}\sum_{i=1}^m\left(a_i^\top Xa_i-y_i\right)^2 \\
\mbox{subject to\ensuremath{\quad}} & \tr(X)=\tau, \quad X\succeq 0.
\end{aligned}
\end{equation}
We set $m = 15n r_\natural$ in all our experiments. 

\paragraph{Low rankness and strict complementarity.} We verify the low rankness 
and strict complementarity for $n=100,200,400$ and $600$. We set $c=0.5$ for the noise Level. 
We also set $\tau =0.5$, since otherwise, the optimal solution will fit the noise and results in a higher rank matrix. 
Problem \eqref{eq: quadraticSensing} is solved via FASTA \cite{GoldsteinStuderBaraniuk:2014, FASTA:2014}. We found 
that every optimal solution rank in this case is $\rsol =3$, and there is indeed a significant gap between 
$\lambda_{n-3}(\nabla f(\xsol))$ and $\lambda_{n}(\nabla f(\xsol))$, which verifies strict 
complementarity. More details can be found in Table \ref{table: LowRankandStricComplementarity}. 
\begin{table} \label{table: LowRankandStricComplementarity}
\centering
\begin{tabular}{ccc} 
 	\toprule
 	Dimension $n$ & Avg. gap  & Avg. recovery error \\ 
 	\midrule
 	100 & 288.06  &0.0013 \\
 	200 & 505.16 & 0.00064 \\
 	400 & 961.09  & 0.00031\\
 	600 & 1358.62 & 0.00021\\
    \bottomrule 
\end{tabular}
\caption{Verification of low rankness and strict complementarity. The 
recovery error is measured by $\frac{\fronorm{\frac{\xsol}{\tau}-U_\natural U_\natural^\top}}{\fronorm{U_\natural U_\natural^\top}}$.
The gap is measured by $\lambda_{n-3}(\nabla f(\xsol))-\lambda_{n}(\nabla f(\xsol))$. All the results is averaged over $20$ iid trials.} 
\end{table}  
\paragraph{Comparison of algorithms.} We now compare the performance of 
FW, G-BlockFW, and SpecFW. We follow the setting as the previous paragraph for 
$n=100,200,400$, and $600$. 
We set $k=4$ 
for both SpecFW and G-BlockFW, which is larger than $\rsol=3$.
We also set $\eta =0.4$ and 
$\beta = 2.5n^2$.\footnote{This choice might 
appear conservative. But we note that $a_i$ has length around $\sqrt{n}$. Hence the operator norm 
of $\Amap$ is around $\sqrt{m}n$ ($\inprod{a_ia_i^\top}{X}\leq \fronorm{a_i}^2\fronorm{X}\approx n\fronorm{X}$), which suggests $L_f=\opnorm{\Amap}^2=n^2m$ as a safe choice. We have 
already omitted one $m$ factor here for better algorithmic performance.} 
The small-scale SDP \eqref{eq: smallksubproblem} is solved via FASTA. 
We plot the relative objective value against both the time and iteration counter in Figure~\ref{fig:riskCover}. We only present the plot for the case of $n=600$ here and those for the other cases  can be found in Section \ref{sec: additionalNumerics} in the Appendix. As can be seen from Figure \ref{fig:riskCover}(a), SpecFW 
converges faster in terms of both the iteration counter and the time. The oscillation in the end 
may be attributed to the sub-problem solver. 
\paragraph{Misspecification of $k$.} We adopt the same setting as before. In this 
experiment, we set $k=2$ for both SpecFW and G-BlockFW, which is less than $\rsol=3$. As can be seen from the Figure \ref{fig:riskCover}(b), 
SpecFW still converges as fast as FW (the two line coincide). G-BlockFW gets stuck around $10^{-1}$ and
stop converging to the optimal solution.

\begin{figure}[t!]
	\begin{center}
		\subfigure[$k> \rsol$]{\includegraphics[width=1\linewidth]{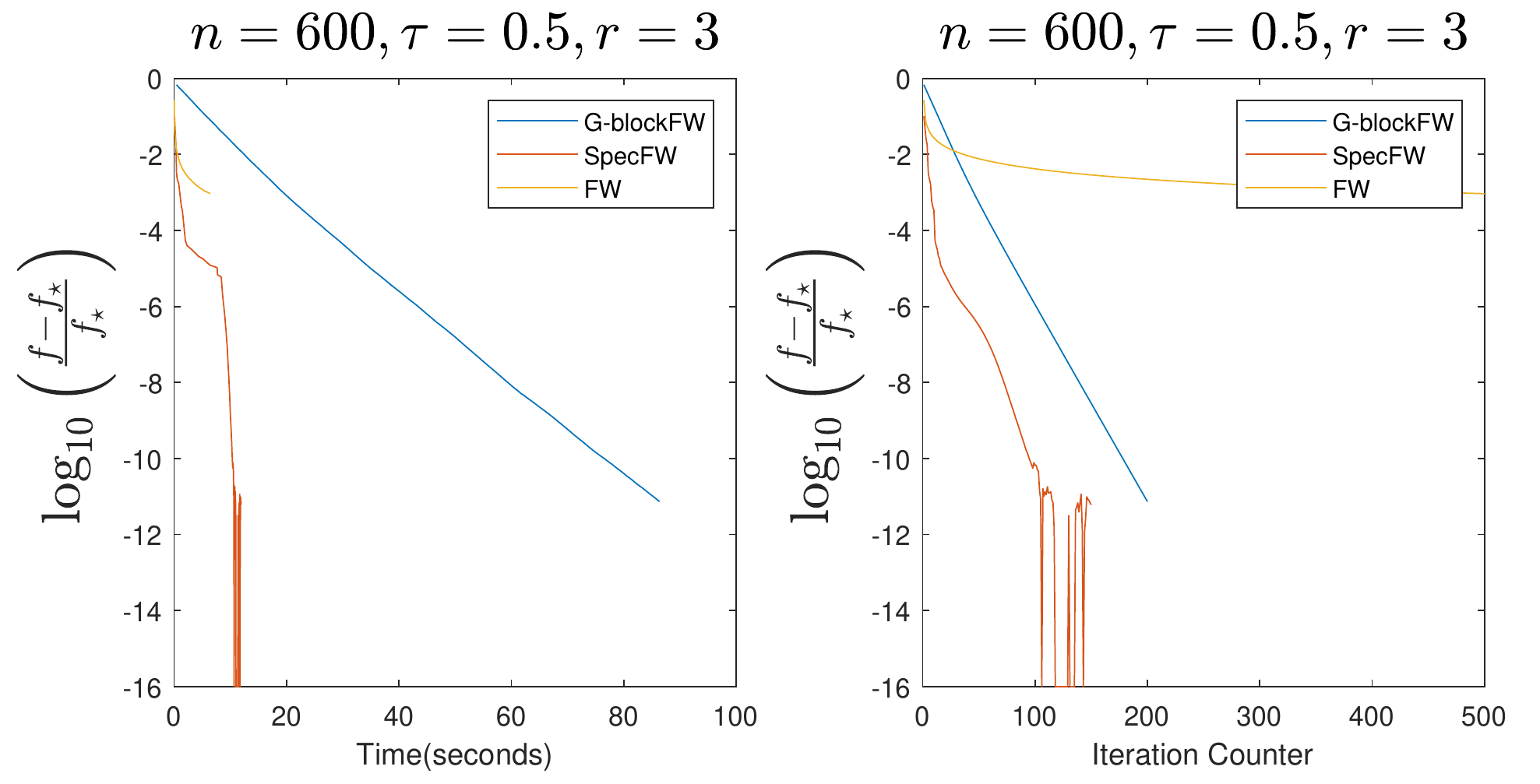}}
		\subfigure[$k< \rsol$]{\includegraphics[width=1\linewidth]{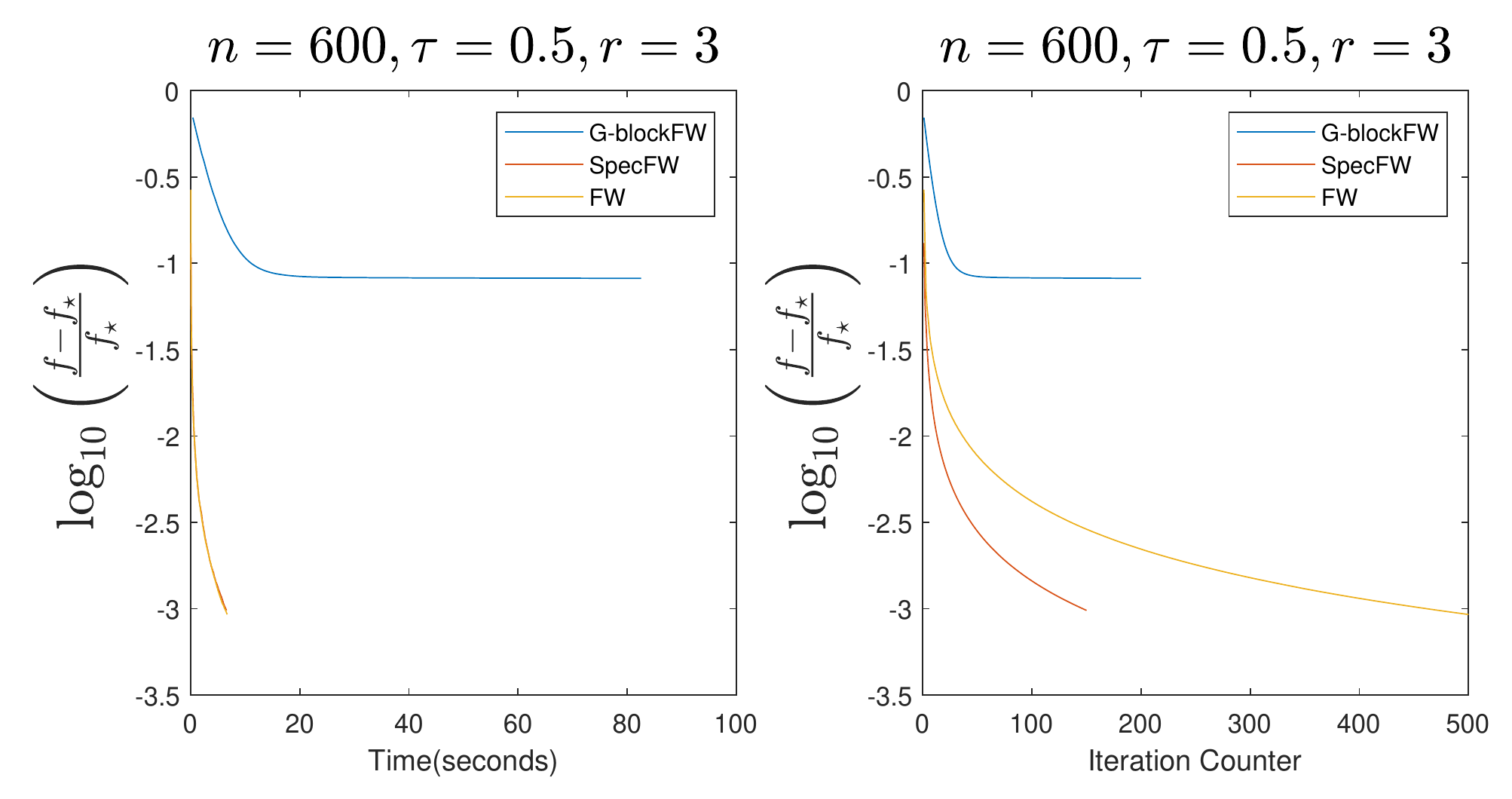} }
	\end{center}
	\caption{Comparison of algorithms under different setting.
	$f^\star$ is obtained from the best value of the three methods and FASTA.}
	\label{fig:riskCover}
\end{figure}
\section{Discussion}
In this paper, we propose the Spectral Frank-Wolfe algorithm, a novel variant of the classical Frank-Wolfe algorithm, which converges sublinearly for convex smooth optimization problems and converges linearly when strict complementary is satisfied for structural convex optimization problems. We also show that the quadratic growth condition, which is essential for linear convergence of first order methods, holds under strict complementarity. 

Here we discuss two potential (and hopefully interesting) extensions of the current paper:
\begin{itemize}
	\item \textbf{Total computational complexity:} The complexity of subproblem \eqref{eq: smallksubproblem} is not discussed and hence leave the total complexity unresolved. Simply using the known $\mathcal{O}(\frac{1}{\sqrt{\epsilon}})$ result of APG for the subproblem complexity seems to be too pessimistic. Is it possible to improve this complexity to $\mathcal{O}(\log(\frac{1}{\epsilon}))$? 
	\item \textbf{Solving Subproblem \eqref{eq: smallksubproblem} by sub-sampling?} In many applications, $f$ is of a finite sum structure with $m$ terms, e.g., matrix completion, and quadratic sensing. The number $m$ is usually on the order $\dm \rsol$. In the subproblem \eqref{eq: smallksubproblem}, the decision variable has size $\mathcal{O}(k^2)$, which can be much smaller than $m$. It might be 
	unwise to use all the $m$ terms. Can we sub-sample the $m$ terms to reduce the burden of computing gradient?
\end{itemize}
\section*{Acknowledgements}
L. Ding and Y. Fei were supported by the National Science Foundation CRII 
award 1657420 and grant 1704828.
C. Yang was supported in part by DARPA Award FA8750-17-2-
0101. We would like to thank Yudong Chen, Madeleine Udell, 
James Renegar, and Adrian Lewis for helpful discussions.
%
%


\bibliography{example_paper}
\bibliographystyle{icml2020}

\input appendix
\end{document}

%% file: appendix.tex
\onecolumn
\icmltitle{Appendices to ``Spectral Frank-Wolfe Algorithm: Strict Complementarity and Linear Convergence"}

\appendix
\section{Uniqueness assumption}\label{sec: uniqueness}
Here we discuss how to adapt our results to multiple solution setting. 
First of all, if there are multiple solution, the strict complementarity condition 
means that there is a primal optimal solution $\xsol$ such that 
\[
\rank(\xsol)+\rank(\zsol)=\dm.
\]
Thus we should set $\rsol$ to be the maximal rank among all primal solutions. 
Denote the set of primal optimal solution of Problem \eqref{eq:Mainoptimization} as 
$\mathcal{X}_\star$. Quadratic growth in this situation is understood as 
\[
f(X)-f(\xsol)\geq \gamma \inf_{\xsol \in \mathcal{X}_\star}\fronorm{X-\xsol}=:\dist(X,\mathcal{X}_\star),
\]
for any $X\succeq 0$ and $\tr(X)=1$.
Now 
due to strict complementarity, we still have $\rsol =\dualdim$ (dual solution $\zsol$
is unique as shown in the next section). Theorem \ref{thm:MainAlgorithmTheorem} can 
be now be proved in the exactly same way by considering the 
nearest $\xsol \in \mathcal{X}_\star$  to $X_t$
without the uniqueness assumption. To prove Theorem \ref{thm:MainStructralTheorem}, 
the argument follows exactly as the main proof by considering the nearest $\xsol\in \mathcal{X}_\star$ to $X$, and replacing Lemma \ref{lem:Suppose-the-following} by Lemma \ref{lem: nonuniquesolutionQGLemma}. In this case, the parameter $\gamma$ of quadratic growth is $\gamma=\min\left\{ \frac{\lambda_{n-\rsol}(\zsol)}{4+8\frac{\sigma_{\max}^{2}(\tilde{\Amap})}{\mu^2}},\frac{\alpha\mu^2}{8}\right\}$ where $\mu:=\sup\{a \geq 0\mid a\cdot \dist(X,\mathcal{X}_\star)\leq 
\twonorm{\tilde{\Amap}(X)-b} \text{ for all } X\in \face{\rsol}(\zsol)\}$ and is indeed positive using Lemma \ref{lem: nonuniquesolutionQGLemma}.

\section{Lemmas for Section \ref{sec: StrictComplementarityAndGrowth}}

\begin{lem}\label{lem: uniquenessOfTheDual}
	The dual solution $(\zsol,\ssol)$ of Problem \eqref{eq:Mainoptimization} is unique even if 
	the primal solution is not unique. 
\end{lem}
\begin{proof}
	We first show that for any primal solution $\xsol$, its gradient 
	$\nabla f(\xsol)$ is the same. 
Using $\beta$-smoothness of $f$ (the constant $\beta$ can be taken to be $\opnorm{\Amap}^2L_g$),
we have for any optimal $\xsol$ and $\xsol'$
\begin{equation}
\begin{aligned} \label{eq: smoothness}
&\inprod{\xsol-\xsol'}{\nabla f(\xsol)-\nabla f(\xsol')}\\
\geq & \frac{1}{\beta}\fronorm{\nabla f(\xsol)-\nabla f(\xsol')}^2.
\end{aligned} 
\end{equation}
Since $\xsol$ and $\xsol'$ are optimal solution, we have the following 
two inequalities using the optimality 
\begin{align}
\inprod{\xsol -\xsol '}{\nabla f(\xsol)}&\leq 0, \label{eq: optimalityConditionForGradientNablafxsol}\\
\inprod{\xsol'-\xsol}{\nabla f(\xsol')}& \leq 0. \label{eq: optimalityConditionForGradientNablafxsol'}
\end{align}
Combining the inequalities \eqref{eq: smoothness}, \eqref{eq: optimalityConditionForGradientNablafxsol},
and \eqref{eq: optimalityConditionForGradientNablafxsol'}, we have 
\begin{equation}
\fronorm{\nabla f(\xsol)-\nabla f(\xsol')}\leq 0\implies f(\xsol)=f(\xsol').
\end{equation}
This shows that $\nabla f(\xsol)$ is unique. Now for any 
$\zsol,\ssol$ and $\zsol',\ssol'$ satisfying the KKT condition, 
we have
\begin{equation}
\begin{aligned}
\nabla f(\xsol)+C & = \zsol +\ssol I \\
 & =\zsol'+\ssol' I\\
 \implies &\zsol-\zsol'=(\ssol'-\ssol)I.  
\end{aligned}
\end{equation}
Now using complementarity in step $(a)$ and 
feasibility of $\xsol$ in step $(b)$:
\begin{equation}
\begin{aligned}\label{eq: contradictionzsolzsol'unqiue} 
0 &\overset{(a)}{=}\inprod{\zsol -\zsol'}{\xsol}
=(\ssol'-\ssol)\inprod{I}{\xsol}\\
&\overset{(b)}{=}(\ssol'-\ssol)\\
\implies &\ssol=\ssol', \quad \text{and}\quad \zsol=\zsol'.
\end{aligned} 
\end{equation}
Hence the dual solution $\zsol$ and $\ssol$ is unique. 
\end{proof}
\begin{lem}\label{lem: genericstrictcomplementarity}
For almost all $C$, the strict complementarity condition holds for 
\eqref{eq:Mainoptimization}.  
\end{lem}
\begin{proof}
Let us first define indicator function: for any given $D\subset \real^\dm$, we define 
\[
\chi_C(x)=\begin{cases}
0, & x\in D\\
+\infty ,& x\not\in D.
\end{cases}
\]
Also denote the relative interior of a set $D$ as $\relint(D)$.  
We utilize the result in \citet[Corollary 3.5]{drusvyatskiy2011generic}, that for almost all $C$, we have 
\begin{equation} 
\begin{aligned}
-C\in& \relint(\partial( g(\Amap X)\\
     & + \chi_{\{\tr(X)=1\}}(X)+\chi_{\{X\succeq 0\}} (X))(\xsol))\\
    \overset{(a)}{=}& \relint(\Amap^*(\nabla g)(\Amap \xsol)+\{sI\mid s\in \real\}\\
     &+\{-Z\mid Z\succeq 0,\range(Z)\subset \nullspace(\xsol)\})\\
    \overset{(b)}{=}&\Amap^*(\nabla g)(\Amap \xsol)+C+\{sI\mid s\in \real\} \\ 
     &+ \{-Z\mid Z\succeq 0,\range(Z)= \nullspace(\xsol)\}.
\end{aligned}
\end{equation}  
Here we use the sum rule in step $(a)$ as $\frac{1}{\dm}I$ is in
$\{X\mid \tr(X)=1\}$ and the interior of $\{X\mid X\succeq 0\}$. In step $(b)$, 
we use the sum rule of relative interior. Hence, there is some $\ssol$ and 
$\zsol$ such that 
\begin{equation}
\begin{aligned} 
  & \range(\zsol) =\nullspace(\xsol) \\
 \implies& \inprod{\zsol}{\xsol}=0,\quad \text{and}\\
 &  \rank(\zsol)+\rank(\xsol)=\dm.
\end{aligned} 
\end{equation}
and 
\[\Amap^*(\nabla g)(\Amap \xsol)+C=\zsol +\ssol I. 
\]
We thus conclude $(\zsol,\ssol)$ satisfies the KKT condition \eqref{eq:KKT}, and 
strict complementarity holds.
\end{proof}
\section{SpecFW: minimizing an upper bound of $f(\eta X_t+VSV^\top)$.}\label{sec: upperboundf}
When the function $f$ is not fully known or gradient might be hard to query, we may consider the following
subproblem instead: solve 
\begin{equation}
\begin{aligned}\label{eq:spectralfrankwolfelimitedknowledge}
\text{minimize} \quad & g(\Amap X_{t}) \\
&+\inprod{\Amap (\eta X_{t}+VSV^{\top})-\Amap X_t}{(\nabla g)(\Amap X_t)}\\
&+\frac{L_{g}}{2}\twonorm{\Amap (\eta X_{t}+VSV^{\top})-\Amap X_t}^{2}\\
& +\inprod{C}{\eta X_{t}+VSV^{\top}}\\
\text{subject to}   
\quad & \eta+\tr(S)=1, \;S\succeq 0,\;\text{and}\;\eta \geq 0.
\end{aligned}
\end{equation}
with decision variable $S$ and $\eta$. Then set $X_{t+1}=  \eta X_{t}+VSV^{\top}$
for the optimal $\eta$ and $S$. 

The above formulation enjoys the advantage of efficient computation in terms 
of time
when $m$ is small and the linear map $\Amap$ and $\inprod C{\cdot}$
are easy to apply to low rank matrices. One may also save $\Amap X_t$ during the 
process to avoid forming $X_t$ and sketching $X_t$ using idea from
\citet{tropp2017practical} for storage purpose. 

One could also consider solving 
\begin{equation}
\begin{aligned}\label{eq:spectralfrankwolfefurtherlimitedknowledge}
\text{minimize} \quad & f(X_t) \\
&+\inprod{\eta X_{t}+VSV^{\top}-X_t}{\nabla f(X_t)} \\
&+\frac{L_f}{2}\fronorm{X_t-(\eta X_{t}+VSV^{\top})}\\
\text{subject to}   
\quad & \eta+\tr(S)=1, \;S\succeq 0,\;\text{and}\;\eta \geq 0.
\end{aligned}
\end{equation}
Then set $X_{t+1}=  \eta X_{t}+VSV^{\top}$
for the optimal $\eta$ and $S$. Here $L_f$ is the 
Lipschitz constant of $\nabla f$. This method requires to store $X_t$ in 
each iteration though. 

\section{Combination with matrix sketching idea in \citet{tropp2017practical}}\label{sec: SpecFWmatrixSketching}
When $m$ is on the order $\dm$, we can employ the matrix sketching 
idea developed in  \citet{tropp2017practical} and \citet{yurtsever2017sketchy} to 
achieve storage reduction. We note 
that if we store $\Amap (X_t)=z_t$ and $c_t = \inprod{C}{X_t}$ at each iteration, then we have no problem in 
doing the small-scale SDP \eqref{eq: smallksubproblem}, as
$f(\eta X_t + VSV^\top)= g(\eta (\Amap X_t)+\Amap(VSV^\top))+\eta \inprod {C}{X_t}
+\inprod{C}{VSV^\top}$. If
$\Amap$ and inner product with $C$ can be applied to low rank matrices efficiently, then 
 updating $z_t$ and $c_t$ is not hard  
due to linearity of our updating scheme $X_{t+1}= \eta X_t+VSV^\top$. 
 
Now we explain how to  omit storing the iterate
$X_{t}$. First, we draw two matrices with independent standard normal  entries
\begin{equation}
\begin{aligned}
\Psi \in \real^ {n \times k} \quad \text{with} \quad k=2r+1; \\
\Phi \in \real ^{l\times \dm} \quad \text{with} \quad l=4r+3;
\end{aligned}\nonumber
\end{equation} 
Here $r$ is chosen by the user. It either represents the estimate of the true rank of the primal solution or the user's computational budget in dealing with larges matrices. 

We use $Y^C_t$ and $Y^R_t$ to capture the column space  and the row space  of $X_t$:
\begin{align}\label{eqn: onetimeSketch}
Y^C_t = X_t\Psi \in \real^{\dm \times k},\qquad  Y^R_t =\Phi X_t \in \real^{l\times n} .
\end{align}
Hence we initially have $Y^C_0=0$ and $Y^R_0=0$.
Notice that SpecFW does not observe matrix $X_t$ directly. Rather, it observes a stream of rank $k$ updates
\[ X_{t+1} = VSV^\top+ \eta X_t,\] 
where $V\in \real^\dm \times k$ and $S\in \symMat^{k}$. 

In this setting, $Y^C_{t+1}$ and $Y^R_{t+1}$ can be directly computed as
\begin{align}
Y^C_{t+1} = VS(V^\top\Psi)+ \eta Y^C_t \in \real^{\dm \times k }, \label{eq:sketch-primal-update1}\\
Y^R_{t+1} = (\Psi V)SV^\top+ \eta Y^R_t\in \real^{l\times n} . \label{eq:sketch-primal-update2}
\end{align}
This observation allows us to form the sketch $Y^C_t$ and $Y^R_t$ from the stream of updates. 

We then reconstruct $X_t$ and get the reconstructed matrix $\hat{X}_t$ by
\begin{align}\label{eqn: reconstructionsketch}
Y^C_t = Q_tR_t, \quad  B_t = (\Phi Q_t)^{\dagger} Y^R_t, \quad \hat{X}_t= Q_t[B_t]_r,
\end{align}
where $Q_tR_t$ is the $QR$ factorization of $Y^C_t$ and $[\cdot]_r$ returns the best rank $r$ approximation in Frobenius norm.
Specifically, the best rank $r$ approximation of a matrix $Z$ is $U\Sigma V^*$,
where $U$ and $V$ are right and left singular vectors corresponding to the $r$
largest singular values of $Z$ and $\Sigma$ is a diagonal matrix with $r$ largest singular values of $Z$. In actual implementation, we 
may only produce the factors $(QU, \Sigma, V)$ defining $\hat{X}_T$ in the end instead of reconstructing $\hat{X}_t$ in every iteration. We refer the reader to \citet[Theorem 5.1]{tropp2017practical} for the theoretical guarantees on the reconstruction matrix $\hat{X}_t$. 

Hence we can avoid the \emph{forming a new iteratre} procedure in SpecFW. We remark that 
the reconstructed matrix $\hat{X}_t$ is not necessarily positive semidefinite. However, this suffices for the purpose of finding a matrices close to $X_t$. More sophisticated procedure is available for producing a positive semidefinite approximation of $X_t$ \citep[Section 7.3]{tropp2017practical}.  


\section{Proofs for Section \ref{sec: SpectralFrank-WolfeAndItsConvergenceRate}}\label{sec: AppendixSpectralFrank-WolfeAndItsConvergenceRate}
We first give the detailed calculation of the derivation for \eqref{eq:secondparthm3}. 
\begin{proof}[Continuation of proof of Theorem \ref{thm:MainAlgorithmTheorem}]
We need to choose $\xi\in[0,1]$ so that  $1-\xi+\frac{\xi^{2}\beta}{\gamma}$ is minimized while 
keeping $\xi^{2}\beta-\frac{\xi\lambda_{n-\rsol}(\zsol)}{6}\leq0$. For
$\xi^{2}\beta-\frac{\xi\lambda_{n-\rsol}(\zsol)}{6}\leq0$, we need
$\xi\le\frac{\lambda_{n-\rsol}(\zsol)}{6\beta}$. The function $q(\xi)=1-\xi+\frac{\xi^{2}\beta}{\gamma}$
is decreasing for $\xi\leq\frac{\gamma}{2\beta}$ and increasing for
$\xi\geq\frac{\gamma}{2\beta}$. If $\frac{\gamma}{2\beta}\leq\frac{\lambda_{n-\rsol}(\zsol)}{6\beta}$,
then we can pick $\xi=\frac{\gamma}{2\beta}$, and $q(\xi)=1-\frac{\gamma}{4\beta}$.
If $\frac{\gamma}{2\beta}\geq\frac{\lambda_{n-\rsol}(\zsol)}{6\beta}\implies\frac{\lambda_{n-\rsol}(\zsol)}{\gamma}\leq3$,
then we can pick $\xi=\frac{\lambda_{n-\rsol}(\zsol)}{6\beta}$, and
$q(\xi)=1-\frac{\lambda_{n-\rsol}(\zsol)}{6\beta}+\frac{\lambda_{n-\rsol}^{2}(\zsol)}{36\gamma\beta}=1+\frac{\lambda_{n-\rsol}(\zsol)}{6\beta}\left(\frac{\lambda_{n-\rsol}(\zsol)}{6\gamma}-1\right)\leq1-\frac{\lambda_{n-\rsol}(\zsol)}{12\beta}$. 
\end{proof}

We shall prove Lemma \ref{lem:Let--and} in this section. We 
restate Lemma \ref{lem:Let--and} in a 
self-contained way. 
\begin{lem}
	\label{lem:Suppose--with}Suppose $Y\in\symMat^{\dm}$ with eigenvalues
	$\lambda_{1}(Y)\geq\dots\geq\lambda_{\dm}(Y)$, and $\lambda_{n-r}(Y)-\lambda_{n-r+1}(Y)\geq\delta$.
	Here $\lambda_{i}(\cdot)$ denote the operator of taking the $i$-th
	largest eigenvalue. Also let $v_{1},\dots,v_{n}$ be the corresponding
	orthornomal eigenvectors. Denote the eigenspace corresponding to the
	last $r$eigenvalus of $Y$ as $\mathcal{V}_{Y,r}$ and the corresponding
	orthorgonal projection $P_{Y,r}:\real^{\dm}\rightarrow\real^{\dm}$
	which is also a matrix in $\real^{\dm\times\dm}$. Let $V_{Y,r}\in\real^{\dm\times r}$
	formed by the last $r$ many eigenvectors $v_{n-r+1},\dots v_{n}$
	which represents the eigensapce $\mathcal{V}_{Y,r}$. Define $\face r(Y)=\left\{ V_{Y,r}SV_{Y,r}^{\top}\mid S\succeq0,\tr(S)=1\right\} .$
	Then for any $X\in\symMat^{\dm}$ with $\tr(X)=1,X\succeq0$, there
	is some $W\in\face r(Y)$ such that 
	\[
	\inprod{X-W}Y\geq\frac{\delta}{2}\fronorm{X-W}^{2}.
	\]
\end{lem}

\begin{rem}
	We note that as long as $\range(V)=\range(V_{Y,r})$ for some matrix
	$V\in\real^{\dm\times r}$ with orthonormal columns, the set $\face r(Y)$
	is the same as $\left\{ VSV^{\top}\mid S\succeq0,\tr(S)=1\right\} $.
\end{rem}

\begin{proof}[Proof of Lemma \ref{lem:Let--and}]
	We first decompose $X$ by 
	\[
	X=\underbrace{\left(X-P_{Y,r}XP_{Y,r}\right)}_{X_{1}}+\underbrace{P_{Y,r}XP_{Y,r}}_{=:X_{2}}.
	\]
	Note that $P_{Y,r}=P_{Y,r}^{\top},$ so $X_{2}=P_{Y,r}XP_{Y,r}$ is
	still symmetric. Let $1-\epsilon=\tr(P_{Y,r}XP_{Y,r})$. Since $\tr(X)=1,$
	we have $\epsilon=\tr(X-P_{Y,r}XP_{Y,r}).$ We have $\epsilon\in[0,1]$
	as $\tr(P_{Y,r}XP_{Y,r})=\inprod{X}{P_{Y,r}P_{Y,r}}\overset{(a)}{\leq }\opnorm{P_{Y,r}}\tr(X)\leq 1$
		where step $(a)$ is due to H\"older's inequality.
	
	Consider the eigenvalue decomposition of $X_{2}=V_{2}\Lambda_{2}V_{2}^{\top},$
	where $V_{2}\in\real^{\dm\times r}$ and $\Lambda_{2}\in\symMat^{r}$
	with all diagonal nonnegative. Here the column space of $V_{2}$ satisfies
	$\range(V_{2})=\mathcal{V}_{Y,r}$. 
	
	Because $P_{Y,r}XP_{Y,r}=X_{2}$ is a member in $\face r(Y)$, we
	know there is an $W\in\face r(Y)$ such that $W=V_{2}\Lambda_{W}V_{2}^{\top}$
	where $\Lambda_{W}\in\symMat^{r}$ has nonegative diagonal with $\tr(\Lambda_{W})=1$
	and the difference matrix $\Delta=\Lambda_{W}-\Lambda_{2}$ has nonnegative
	entries. We also have $\tr(\Delta)=\epsilon,$ as the trace of both
	$\Lambda_{W}$ and $X$ are one. 
	
	With such choice of $W$, let us now analyze $\inprod{X-W}Y:$
	\begin{equation} 
	\begin{aligned}\label{eq:decompositioninnerproductXWYfirstStep}
	\inprod{X-W}Y & =\inprod{X_{1}}Y+\inprod{X_{2}-W}Y\\
	& =\underbrace{\inprod{X-P_{Y,r}XP_{Y,r}}{\sum_{i=1}^{n}\lambda_{i}(Y)v_{i}v_{i}^{\top}}}_{R_1}\\
	&-\underbrace{\inprod{V_{2}\Delta V_{2}^{\top}}{\sum_{i=1}^{n}\lambda_{i}(Y)v_{i}v_{i}^{\top}}}_{R_2}.
	\end{aligned}
	\end{equation}

	The first term $R_1=\inprod{X-P_{Y,r}XP_{Y,r}}{\sum_{i=1}^{n}\lambda_{i}(Y)v_{i}v_{i}^{\top}}$
	satifies 
	\begin{align*}
	&\inprod{X-P_{Y,r}XP_{Y,r}}{\sum_{i=1}^{n}\lambda_{i}(Y)v_{i}v_{i}^{\top}} \\
       \overset{(a)}{=} &\sum_{i=1}^{n}\lambda_{i}(Y)v_{i}^{\top}Xv_{i}-\sum_{i=n-r+1}^{n}\lambda_{i}(Y)v_{i}^{\top}Xv_{i}\\
	                  = & \sum_{i=1}^{n-r}\lambda_{i}(Y)v_{i}^{\top}Xv_{i}\\
	 \overset{(b)}{\geq}&(\lambda_{n-r+1}(Y)+\delta)\sum_{i=1}^{n-r}v_{i}^{\top}Xv_{i}.
	\end{align*}
	Here in step $(a)$ we uses the fact that $P_{Y,r}v_{i}=v_{i}$ for
	$i=n-r+1,\dots n$ and is zero for other $v_{i}.$ In step $(b)$,
	we use the assumption that $\lambda_{n-r}-\lambda_{n-r+1}\geq\delta$
	and each $v_{i}^{\top}Xv_{i}\geq0$ as $X\succeq0$. We note that
	$\sum_{i=1}^{n-r}v_{i}^{\top}Xv_{i}$ satifies 
	\begin{align*}
	\sum_{i=1}^{n-r}v_{i}^{\top}Xv_{i} & =\tr\left(X\left(\sum_{i=1}^{n-r}v_{i}v_{i}^{\top}\right)\right)\overset{(a)}{=}\tr(X(I-P_{Y,r}))\\
	&\overset{(b)}{=}\tr(X)-\tr(P_{Y,r}XP_{Y,r})=\epsilon.
	\end{align*}
	Here step $(a)$ uses the $P_{Y,r}=V_{Y,r}V_{Y,r}^{\top}$ and we
	use $P_{Y,r}^{2}=P_{Y,r}$ and cyclic property of trace in step $(b).$ 
	
	Now let us analyze the second term $R_2$:
	\begin{align*}
	R_2 & =\inprod{V_{2}\Delta V_{2}^{\top}}{\sum_{i=1}^{\dm}\lambda_{i}(Y)v_{i}v_{i}^{\top}}\\
	& \overset{(a)}{=}\inprod{V_{2}\Delta V_{2}^{\top}}{\sum_{i=n-r+1}^{\dm}\lambda_{i}(Y)v_{i}v_{i}^{\top}}.
	\end{align*}
	Here we use the fact that $V_{2}^{\top}v_{i}=0$ for all $v_{i},i=1,\dots n-r$.
	Since $V_{Y,r}$ and $V_{2}$ are both orthonormal representation
	of $\mathcal{V}_{Y,r}$, we know there is an orthonormal matrix $O\in\real^{r\times r}$
	such that $V_{Y,r}=V_{2}O$. Define the linear operator $\diag:\symMat^{\dm}\rightarrow \real^\dm$
	, which takes the diagonal of a matrix. Let $\Lambda_{Y,r}=\diag ^*\left(\lambda_{n-r+1}(Y),\dots,\lambda_{n}(Y)\right),$
	we see $R_2$ further equals to 
	\begin{align*}
	R_2 & =\tr\left(V_{2}\Delta V_{2}^{\top}V_{2}O\Lambda_{Y,r}O^{\top}V_{2}^{\top}\right)\\
	& \overset{(a)}{=}\tr\left(\Delta O\Lambda_{Y,r}O^{\top}\right)\\
	& \overset{(b)}{\leq}\epsilon\lambda_{n-r+1}(Y).
	\end{align*}
	Here we use the cyclic property in step $(a)$ and the step $(b)$
	is an easy consequence of $\Delta$ has nonnegative diagonal and Von
	Neumann's trace inequality: for symmetric matrices $A,B\in\symMat^{r},\inprod AB\leq\sum_{i=1}^{r}\lambda_{i}(A)\lambda_{i}(B)$.
	Combining pieces, we find that 
	\[
	\inprod{X-W}Y\geq(\lambda_{n-r+1}(Y)+\delta)\epsilon-\epsilon\lambda_{n-r+1}(Y)=\delta\epsilon.
	\]
	Now we turn to analyzing the term $\fronorm{X-W}^{2}$. Using $\inprod{X_{1}}{X_{2}}=0,\inprod{X_{1}}W=0,$
	we find that
	\[
	\fronorm{X-W}^{2}=\fronorm{X_{1}}^{2}+\fronorm{X_{2}-W}^{2}.
	\]
	The second term $\fronorm{X_{2}-W}^{2}$ satisfies 
	\[
	\fronorm{X_{2}-W}=\fronorm{V_{2}\Delta V_{2}^{\top}}^{2}=\sum_{i=1}^{r}\Delta_{ii}^{2}\leq\left(\sum_{i=1}^{r}\Delta_{ii}\right)^{2}=\epsilon^{2}.
	\]
	If we write $X$ in terms of the coordinates given by $V_{2}$ and
	its orthogonal compliment say $V_{1}$, then in this new coordinate
	$V=[V_{1},V_{2}]$:
	\[
	V^\top XV=\begin{bmatrix}A & B\\
	B & V_2^\top X_{2}V_2
	\end{bmatrix},\quad\text{and}\quad V^\top X_{1}V=\begin{bmatrix}A & B\\
	B & 0
	\end{bmatrix}.
	\]
	Then $\tr(X_{1})=\tr(A)$. Lemma \ref{lemma: auxSecOracleinequalities} 
	implies that 
	\[
	\fronorm B^{2}\leq\tr(X_{2})\tr(A)=\epsilon(1-\epsilon)=\epsilon-\epsilon^{2}.
	\]
	Hence $\fronorm{X_{1}}^{2}=\fronorm A^{2}+2\fronorm B^{2}\leq\left(\tr(A)\right)^{2}+2\epsilon-2\epsilon^{2}=-\epsilon^{2}+2\epsilon.$
	Combining pieces and $\epsilon\in[0,1]$, we find that 
	\begin{align*}
	\fronorm{X-W}^{2} & \leq2\epsilon=\frac{2}{\delta}\delta\epsilon\leq \frac{2}{\delta } \inprod{X-W}Y\\
	\implies & \inprod{X-W}Y\geq\frac{\delta}{2}\fronorm{X-W}^{2}.
	\end{align*}
	
\end{proof}
\begin{lem} \label{lemma: auxSecOracleinequalities}
	Suppose
	$Y = \begin{bmatrix}
	A & B \\ B^\top   & D
	\end{bmatrix} \succeq 0$.
	Then $ \opnorm{A} \tr(D) \geq \nucnorm{BB^\top}=\tr(BB^\top)=\fronorm{B}^2.$
\end{lem}
\begin{proof}
	For any $\epsilon>0$, denote  $A_\epsilon = A + \varepsilon I$ and
	$ Y_\epsilon = \begin{bmatrix}
	A_\epsilon & B \\ B^*  & D
	\end{bmatrix}.$  We know $Y_\epsilon$ is psd,
	as is its Schur complement
	$D - B^\top  A_\epsilon ^{-1}B \succeq 0$ with trace $\tr(D) - \tr({A}_\epsilon^{-1} BB^\top ) \geq 0.$
	
	Von Neumann's lemma for $A_\epsilon$, $BB^{\top}\succeq 0$
	shows $\tr(A_\epsilon^{-1} BB^* ) \geq \dfrac{1}{\opnorm{A_{\epsilon}}}\nucnorm{BB^\top }$.
	Use this with the previous inequality to see
	$\tr(D) \geq \frac{1}{\opnorm{A_\epsilon}}\nucnorm{BB^\top}.$
	Multiply by $\opnorm{A_{\epsilon}}$ and let $\varepsilon \to 0 $
	to complete the proof. 
\end{proof}
\section{Lemmas for Section \ref{sec: QuadraticGrowthAndLinearConvergence}}\label{sec: lemmaForSection4}
We first give a self-contained proof for the second case of Theorem \ref{thm:MainStructralTheorem}.
\begin{proof}[Proof of second case of Theorem \ref{thm:MainStructralTheorem}]
For any feasible $X$ and
the optimal solution $\xsol,$ we have 
\begin{align*}
f(X)-f(\xsol) & \overset{(a)}{\geq}\inprod{\nabla f(\xsol)}{X-\xsol}\\
& \overset{(b)}{=}\inprod{\zsol+\ssol I}{X-\xsol}\\
& \overset{(c)}{=}\inprod{\zsol}{X-\xsol}.
\end{align*}
Here step $(a)$ is due to the convexity of $f$. For step
$(b)$, we uses the first order condition of KKT condition (\ref{eq:KKT}).
The step $(c)$ is due to feasibility of $X$ and $\xsol.$ 

Since $\zsol$ has rank $n-1$, using strict complementarity, we reach
that any optimal solution $\xsol$ has rank $1$ with $\range(\xsol)=\nullspace(\zsol)$.
Thus any optimal solution $\xsol$ is of the form $\xsol=\xi vv^{\top}$,
$v$ is the non-zero unit vector in the null space of $\zsol$, and
$\xi$ is a nonnegative scaler. Since $\xsol$ has to be feasible,
the constraint $\tr(\xsol)=1$ implies that $\xi=1$ and hence the
solution $\xsol$ is unique. The same argument implies that the set
$\face 1(\zsol)=\left\{ \xsol\right\} .$ Hence using Lemma \ref{lem:Let--and} and
$\lambda_{n}(\zsol)=0$, we see that 
\[
f(X)-f(\xsol)\geq\inprod{\zsol}{X-\xsol}\geq\frac{\lambda_{n-1}(\zsol)}{2}\fronorm{X-\xsol}^{2}.
\]
\end{proof}
Next, we establish the lemma that is core to the proof of Theorem \ref{thm:MainStructralTheorem} under the assumption of uniqueness. 

\begin{lem}
	\label{lem:Suppose-the-following}Suppose the following system admits
	a unique solution $\xsol$ with rank $\rsol:$ 
	\begin{equation}
	\inprod{\zsol}{\xsol}=0,\Amap X=b,\quad\text{and\ensuremath{\quad X\succeq0},}\label{eq: linearmatrixsystem}
	\end{equation}
	for a $\zsol\succeq0$ such that $\rank(\zsol)+\rank(\xsol)=\dm$,
	a linear map $\Amap:\symMat^{\dm}\rightarrow\real^{\cons}$, and a
	vector $b\in\real^{\cons}$. Furthur suppose that $\Amap X=b\implies\tr(X)=1.$
	Then for any $X\succeq0$ with $\tr(X)=1$, we have 
	\begin{equation}\label{eq: X-xsolZXAX-Axsol}
	\begin{aligned}
	\fronorm{X-\xsol}^{2}\;\leq\left(4+8\frac{\sigma_{\max}(\Amap)}{\sigma_{\min}(\Amap_{V})}\right)\frac{\inprod{\zsol}X}{\lambda_{n-\rsol}(\zsol)}\\
	+\frac{4}{\sigma_{\min}^{2}(\Amap_{V})}\twonorm{\Amap(X)-b}^{2}.
	\end{aligned}
	\end{equation}
\end{lem}

\begin{proof}
	Let $V\in\mathbb{R}^{n\times\rsol}$ be a matrix with orthonormal
	columns correpsonding to the eigenspace $\mathcal{V}$ of $\xsol$
	of positive eigenvalues. Then $\xsol$ can be written as $\xsol=VS_{\star}V^{\top}$
	for some $S_{\star}\in\symMat^{\rsol}$ such that $S_{\star}\succ0.$
	We claim that the linear map $\Amap_{V}$ defined as follows is injective:
	\begin{align*}
	\Amap_{V} & :\symMat^{\rsol}\rightarrow\real^{\cons}\\
	& S\mapsto\Amap(VSV^{\top}).
	\end{align*}
	Suppose not, then there is some nonzero $S_{0}\in\symMat^{\rsol}$
	such that $\Amap_{V}(S_{0})=0$. Then $V(\alpha S_{0}+S_{\star})V^{\top}$also
	satisfies the system (\ref{eq: linearmatrixsystem}) for all small
	enough $\alpha$. Hence we see that for any $S\in\symMat^{r}$
	\begin{equation}\label{eq:firstinequalityqudraticgrowthlemma}
	\begin{aligned} 
	\fronorm{VSV^{\top}-\xsol}\leq&\frac{1}{\sigma_{\min}(\Amap_{V})}\twonorm{\Amap(VSV^{\top})-\Amap(\xsol)}\\
	=&\frac{1}{\sigma_{\min}(\Amap_{V})}\twonorm{\Amap(VSV^{\top})-b}.
	\end{aligned}
	\end{equation}
	Here $\sigma_{\min}(\Amap_{V})=\min_{\fronorm S=1}\twonorm{\Amap_{V}(S)}>0$. 
	
	Using strict complementarity on $\zsol$ and $\xsol$, we know $V$ is also a representation
	of the null space of the $\zsol.$ Using Lemma \ref{lem:Let--and},
	we know there is some $W=VSV^{\top}\in\face{\rsol}(\zsol)$ such that
	\begin{equation}
	\inprod X{\zsol}\overset{(a)}{=}\inprod{X-W}{\zsol}\geq\frac{\lambda_{n-\rsol}(\zsol)}{2}\fronorm{X-W}^{2},\label{eq:quadraticgrowthlemmaZero}
	\end{equation}
	where step $(a)$ is because $\lambda_{n-\rsol+1}(\zsol)=\dots=\lambda_{\dm}(\zsol)=0$.
	We note if $\rsol=1$, then $\face r(\zsol)$ has $\xsol$ as its
	only element, as $\tr(X)=1$ and we are done. 
	
	We can bound $\fronorm{X-\xsol}^{2}$ by 
	\begin{align}
	\fronorm{X-\xsol}^{2} & \overset{(a)}{\leq}2\fronorm{X-W}^{2}+2\fronorm{W-\xsol}^{2}\label{eq:secondinequalityqudraticgrowthlemma-revised}\\
	& \overset{(b)}{\leq}2\fronorm{X-W}^{2}+\frac{2}{\sigma_{\min}^{2}(\Amap_{V})}\twonorm{\Amap(W)-b}^{2}.\nonumber 
	\end{align}
	Here we use triangle inequality and basic inequality $(a+c)^{2}\leq2a^{2}+2c^{2}$
	for any real $a,c$ in step $(a)$. In step $(b)$, we use (\ref{eq:firstinequalityqudraticgrowthlemma}). 
	
	We can further bound the term $\twonorm{\Amap(W)-b}$ by 
	\begin{equation}\label{eq:thirdinequalityqudraticgrowthlemma-revised}
	\begin{aligned}
	\twonorm{\Amap(W)-b} &=\twonorm{\Amap(W-X)+\Amap(X)-b}\\
	&\leq\twonorm{\Amap(W-X)}+\twonorm{\Amap(X)-b}.
	\end{aligned} 
	\end{equation}
	Now combining (\ref{eq:secondinequalityqudraticgrowthlemma-revised}),
	(\ref{eq:thirdinequalityqudraticgrowthlemma-revised}) and $(a+c)^{2}\leq2a^{2}+2c^{2}$
	for any $a,c\in\real$ in the following step $(a)$, we see 
	\begin{equation}\label{eq:fourthinequalityqudraticgrowthlemma-revised} 
	\begin{aligned}
	\fronorm{X-\xsol}^{2} & \overset{(a)}{\leq}2\fronorm{X-W}^{2}+\frac{4\twonorm{\Amap(W-X)}^{2}}{\sigma_{\min}^{2}(\Amap_{V})}\\
	&+\frac{4}{\sigma_{\min}^{2}(\Amap_{V})}\twonorm{\Amap(X)-b}^{2}\\
	& \leq\left(2+4\frac{\sigma_{\max}^{2}(\Amap)}{\sigma_{\min}^{2}(\Amap_{V})}\right)\fronorm{X-W}^{2}\\
	&+\frac{4}{\sigma_{\min}^{2}(\Amap_{V})}\twonorm{\Amap(X)-b}^{2}.\nonumber 
	\end{aligned}
    \end{equation} 
	Finally using (\ref{eq:quadraticgrowthlemmaZero}) to bound $\fronorm{X-W}$,
	we reached the inequality we want to prove:
	\begin{equation*} 
	\begin{aligned}
	\fronorm{X-\xsol}^{2} &\leq\left(4+8\frac{\sigma_{\max}^{2}(\Amap)}{\sigma_{\min}^{2}(\Amap_{V})}\right)\frac{\inprod{\zsol}X}{\lambda_{n-\rsol}(\zsol)}\\
	&+\frac{4}{\sigma_{\min}^{2}(\Amap_{V})}\twonorm{\Amap(X)-b}^{2}.
	\end{aligned} 
	\end{equation*}
\end{proof}

We now establish a lemma to handle the general case that the solution might not be unique. 
For a convex closed set $\mathcal{X}_\star$, we define the distance to for an arbitrary $X\in \symMat^{\dm}$ to it as 
\[
\dist(X,\mathcal{X}_\star):\,= \inf_{\xsol \in \mathcal{X}_\star}\fronorm{X-\xsol}.
\]

\begin{lem}
	\label{lem: nonuniquesolutionQGLemma} Denote the solution set of the following system as $\mathcal{X}_\star$:
	\begin{equation}
	\inprod{\zsol}{\xsol}=0,\Amap X=b,\quad\text{and\ensuremath{\quad X\succeq0},}\label{eq: linearmatrixsystemNonuniqueSolution}
	\end{equation}
	for a $\zsol\succeq0$, a linear map $\Amap:\symMat^{\dm}\rightarrow\real^{\cons}$, and a
	vector $b\in\real^{\cons}$. 
	 Suppose the system \eqref{eq: linearmatrixsystem} admits a solution $\xsol^0$ with rank $\rsol^0\geq 1$ such that  $\rank(\zsol)+\rank(\xsol^0)=\dm$. 
	 Further suppose that $\Amap X=b\implies\tr(X)=1.$ 
	Then the constant $\mu:=\sup\{a \geq 0\mid a\cdot \dist(X,\mathcal{X}_\star)\leq 
\twonorm{\Amap(X)-b} \text{ for all } X\in \face{\rsol}(\zsol)\}$ 
	is positive, and
	for any $X\succeq0$ with $\tr(X)=1$, we have 
	\begin{equation}\label{eq: X-xsolZXAX-AxsolNotUniqueSolution}
	\begin{aligned}
	\dist(X,\mathcal{X}_\star)^{2}\;\leq\left(4+8\frac{\sigma_{\max}(\Amap)}{\mu}\right)\frac{\inprod{\zsol}X}{\lambda_{n-\rsol}(\zsol)}\\
	+\frac{4}{\mu^{2}}\twonorm{\Amap(X)-b}^{2}.
	\end{aligned}
	\end{equation}
\end{lem}

\begin{proof}
	Let $V\in\mathbb{R}^{n\times\rsol}$ be a matrix with orthonormal
	columns corresponding to the eigenspace $\mathcal{V}$ of $\rsol$
	zero eigenvalues. Consider the linear map $\Amap_{V}$:
	\begin{align*}
	\Amap_{V} & :\symMat^{\rsol}\rightarrow\real^{\cons}\\
	& S\mapsto\Amap(VSV^{\top}).
	\end{align*}
	The key replacement of multiple solution setting is to establish an inequality similar to 
	\eqref{eq:firstinequalityqudraticgrowthlemma}, which depicts the injectivity of$\Amap_V$ for unique solution setting. 
	
	Define the solution set $\mathcal{S}\subset \symMat^{\rsol}$ of the following system:
	\begin{equation} \label{eq: reducedSystem}
	\Amap_{V}(S)=b, \quad S\succeq 0. 
	\end{equation}
	Note that any $S\in \mathcal{S}$ satisfies that $VSV^\top \in \mathcal{X}_\star$. Conversely, for any $\xsol\in \mathcal{X}_\star$, it can be written as $\xsol=VS_{\star}V^{\top}$
	for some $S_{\star}\in\symMat^{\rsol}$ such that $S_{\star}\succeq0$ and $\Amap_{V}(S_\star)=b$. Hence we have $\mathcal{X}_\star = \{X\mid X=VSV^\top, \; S\in \mathcal{S} \}$. 
	
	Now if we take the $\xsol^0\in \mathcal{X}_\star$ such that $\rank(\zsol)+\rank(\xsol^0)=\dm$,
	then $\xsol^0=VS_{\star}^0V^{\top}$
	for some $S_{\star}^0\in\symMat^{\rsol}$ such that $S_{\star}^0\succ 0.$ This means the system \eqref{eq: reducedSystem} satisfies the condition in Corollary 3 in \cite{bauschke1999strong}. By applying this corollary to \eqref{eq: reducedSystem}, we know there is a 
	$\mu >0$ such that for all $S\succeq 0$ and $\tr(S)=1$,
	\begin{equation}
	    \dist(S,\mathcal{S})\leq \frac{1}{\mu} \twonorm{\Amap_V{S}-b}.
	\end{equation}
	Translating the inequality to the space $\mathcal{L}=\{X\in \symMat\mid X=VSV^\top \text{ for some }S\in \symMat^{\rsol}\}$, we have for all $X\succeq 0$, $\tr(X)=1$, and $X\in \mathcal{L}$, i.e., $X\in \face{\rsol}(\zsol)$ :
	\begin{equation}
	    \dist(X,\mathcal{X}_\star)\leq  \frac{1}{\mu}\twonorm{\Amap(X)-b}.
	\end{equation}
	This is our replacement of \eqref{eq:firstinequalityqudraticgrowthlemma} in Lemma \ref{lem:Suppose-the-following}. 
	
	Following the proof of Lemma \ref{lem:Suppose-the-following}, 
	we know there is some $W=VSV^{\top}\in\face{\rsol}(\zsol)$ such that
	\begin{equation}
	\inprod X{\zsol}=\inprod{X-W}{\zsol}\geq\frac{\lambda_{n-\rsol}(\zsol)}{2}\fronorm{X-W}^{2}.\label{eq:quadraticgrowthlemmaZeroNotUniqueSolution}
	\end{equation}
	
	To bound $\dist(X,\mathcal{X}_\star)$, we pick an $\xsol\in \mathcal{X}_\star$ such that it is nearest to $W$ (mote $\mathcal{X}_\star$ is compact as $\Amap(X)=b$ implies $\tr(X)=1$). Then we have 
	\begin{align}
	\dist(X,\mathcal{X}_\star)^{2}  
	&\leq \fronorm{X-\xsol}^2\\
	& \overset{(a)}{\leq}2\fronorm{X-W}^{2}+2\fronorm{W-\xsol}^{2}\label{eq:secondinequalityqudraticgrowthlemma-revisedNotUniqueSolution}\\
	& \overset{(b)}{\leq}2\fronorm{X-W}^{2}+\frac{2}{\mu^2}\twonorm{\Amap(W)-b}^{2}.\nonumber 
	\end{align}
	Here we use triangle inequality and basic inequality $(a+c)^{2}\leq2a^{2}+2c^{2}$
	for any real $a,c$ in step $(a)$. In step $(b)$, we use \eqref{eq:firstinequalityqudraticgrowthlemma}. The rest of the proof 
	is exactly the same as those in Lemma \ref{lem:Suppose-the-following}.
\end{proof}

The following Lemma establishes the linear convergence of G-BlockFW under quadratic growth condition. 
\begin{lem}\label{lem: linearConvergenceOfBlockFW}
Suppose $f$ of Problem \eqref{eq:Mainoptimization} is $\beta$ smooth and Problem \eqref{eq:Mainoptimization} satisfies quadratic growth with parameter $\gamma$. If $\eta = \frac{\gamma}{\beta} $ and $k\geq \rsol =\rank(\xsol)$, where $\xsol$ is an optimal solution 
of  Problem \eqref{eq:Mainoptimization},  then the generalized 
Block FW \ref{alg:generalizedBlockFW} converges linearly: 
\[
h_{t+1}\leq (1-\frac{\gamma}{2\beta})h_t,
\]
where $h_t = f(X_t)-f(\xsol)$  for each $t$. 
\end{lem}
\begin{proof}
	Denote $\hat{Y} = V\diag(\Lambda)V^\top$. The Lipschitz smoothness of $f$ shows that 
	\begin{equation}
	\begin{aligned} \label{eq: blockFWFirstInequality}
	f(X_{t+1})\leq f(X_t)+\eta\inprod{\hat{Y}-X_t}{\nabla f(X_t)}+\frac{\eta^2\beta }{2}\fronorm{\hat{Y}-X_t}^2.
	\end{aligned} 
	\end{equation}
	Using a similar argument as \citet[Lemma 3.1]{allen2017linear}, 
	we have 	
	\[\hat{Y} =\arg\min_{Y\in \specsplex, \rank(Y)\leq \rsol }
	\eta\inprod{\hat{Y}-X_t}{\nabla f(X_t)}+\frac{\eta^2\beta }{2}\fronorm{\hat{Y}-X_t}^2.\]
	Hence, we can replace $\hat{Y}$ in \eqref{eq: blockFWFirstInequality} by $\xsol$ in the following step (a),
		\begin{equation}
	\begin{aligned} \label{eq: blockFWSecondInequality}
	f(X_{t+1})
	&\overset{(a)}{\leq} f(X_t)+\eta\inprod{\xsol-X_t }{\nabla f(X_t)}+\frac{\eta^2\beta }{2}\fronorm{\xsol -X_t}^2\\
	&\overset{(b)}{\leq} f(X_t)-\eta(f(X_t)-f(\xsol)+\frac{\eta^2\beta}{2\gamma}(f(X_t)-f(\xsol)),
	\end{aligned} 
	\end{equation}
where step $(b)$ is due to the qudratic growth of Problem \eqref{eq:Mainoptimization}. Now subtract both sides by $f(\xsol)$, and let $h_t = f(X_t)-f(\xsol)$ for each $t$, we find that 
\[
h_{t+1}\leq (1-\eta+\frac{\eta^2\beta}{2\gamma})h_t.
\]
Our choice $\eta = \frac{\gamma }{\beta}$ set 
$(1-\eta+\frac{\eta^2\beta}{2\gamma})=1-\frac{\gamma}{2\beta}$ which is what we desired. 
\end{proof}

\section{Additional Numerics}\label{sec: additionalNumerics}
We include extra numerics for $n=100,200,400$ in Figure \ref{fig:correctK}, \ref{fig:missK}. 
As can be seen, SpecFW in these cases are a bit slower than G-BlockFW when $\tau=0.5$ and $c=0.5$. SpecFW is as good as 
FW when $k$ is miss specified. 

\paragraph{What if $\nabla f(\xsol)=0$?} Here we also discuss an interesting situation that $c=0$, and $\tau=1$, then we see 
$\xsol = U_\natural U_\natural^\top$ is an optimal solution and gradient in this case is $0$. 
Such situation means strict complementarity fails and the small perturbation to $\tau$ will 
result in a higher-rank solution, meaning the convex relaxation \eqref{eq: quadraticSensing}
is ill-posed for the purpose of low-rank matrix recovery [Lemma 2]\cite{garber2019linear}. Indeed,
this is where SpecFW is not advantageous comparing to G-BlockFW as shown in Figure \ref{fig:Nonoise}.
$\tau=1$ and $c=0$. 
\begin{figure}[t!]
	\begin{center}
		\subfigure[$n=100$]{\includegraphics[width=0.7\linewidth]{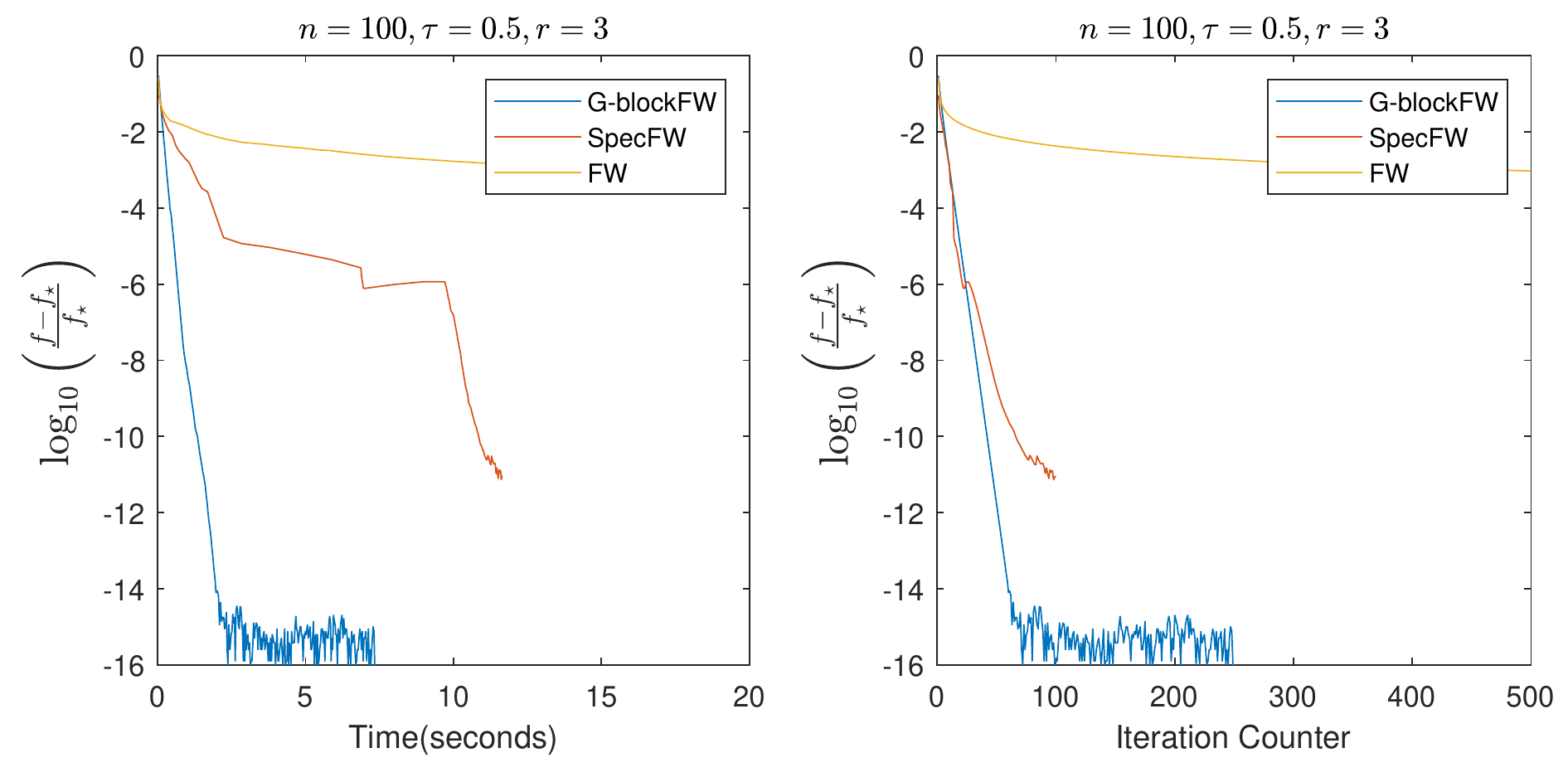}}
		\subfigure[$n=200$]{\includegraphics[width=0.7\linewidth]{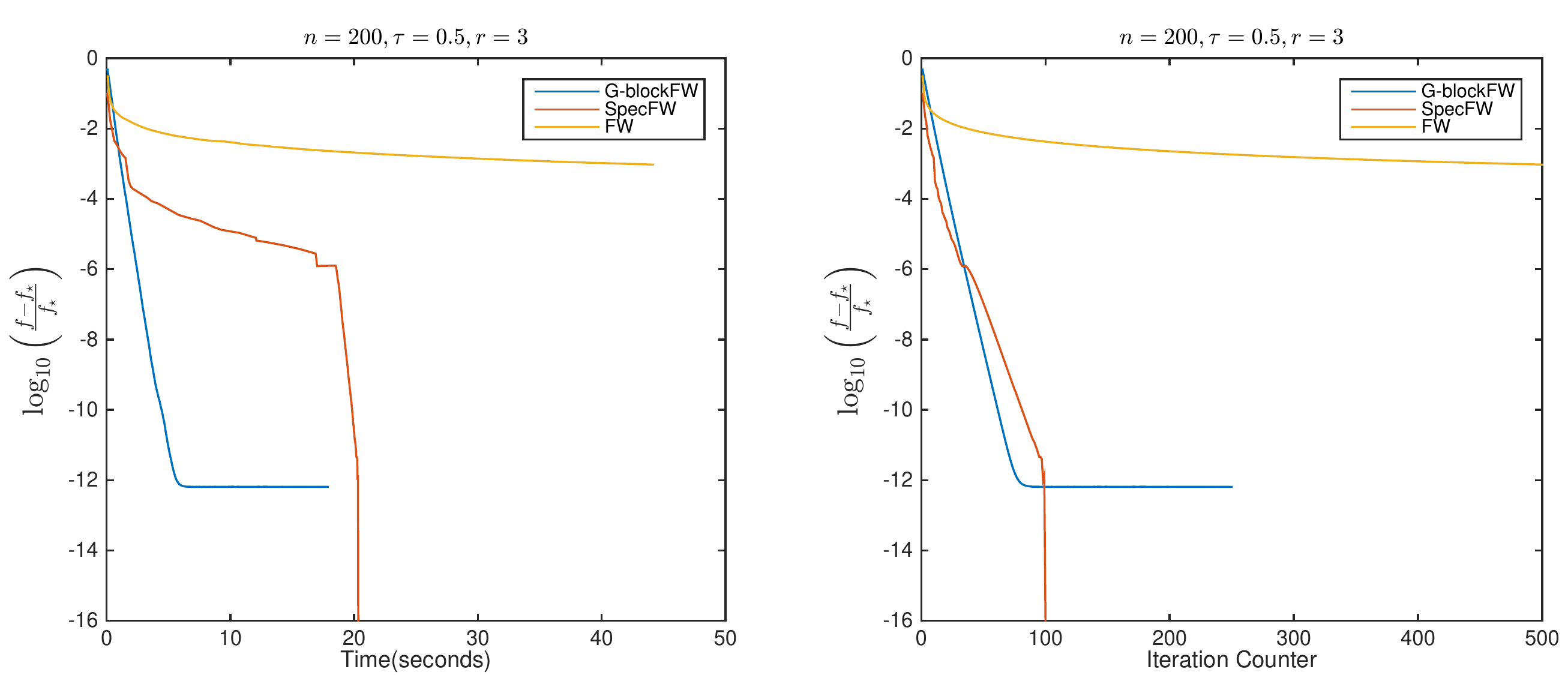} }
			\subfigure[$n=400$]{\includegraphics[width=0.7\linewidth]{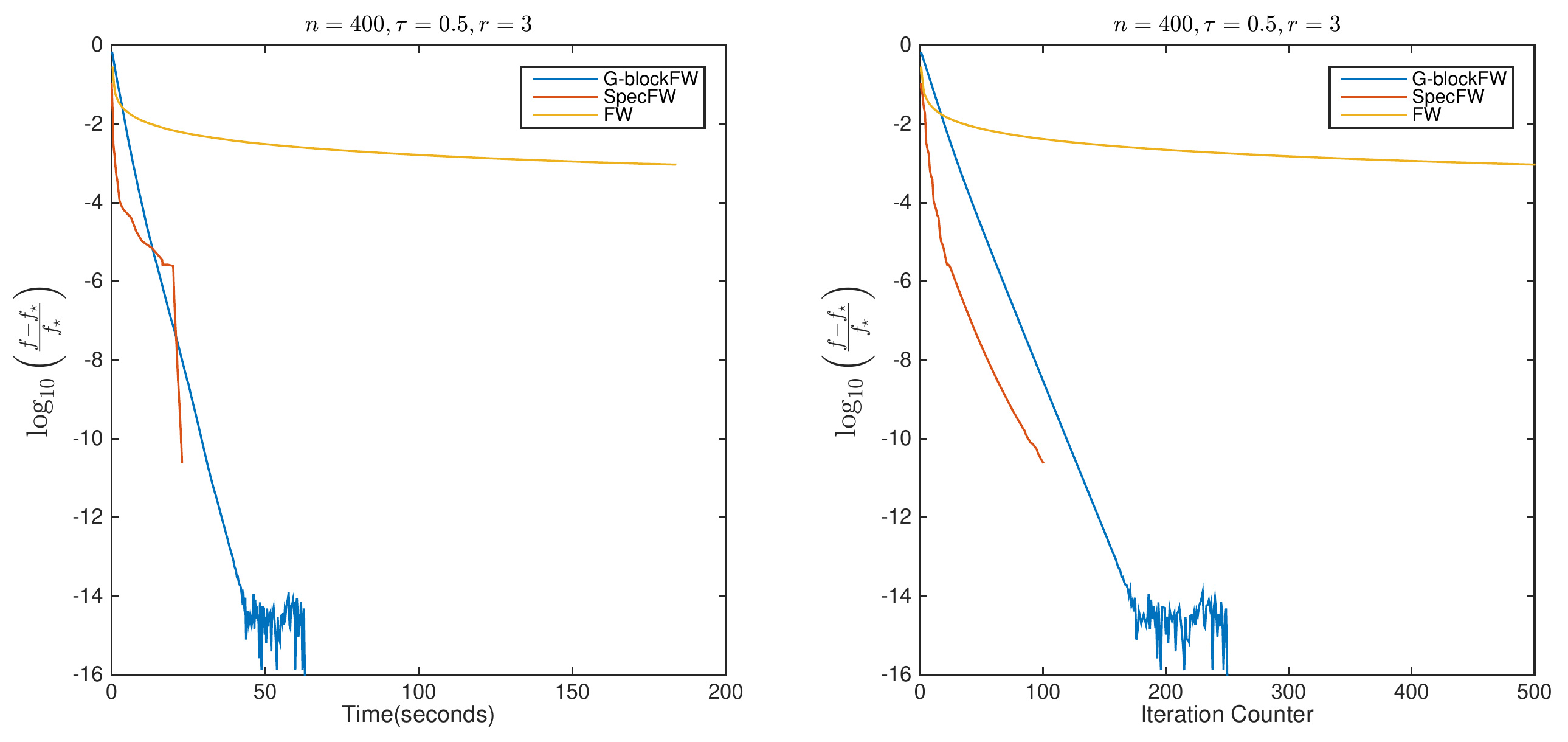} }
	\end{center}
	\caption{Comparison of algorithms under $\tau=\frac{1}{2}$ and noise level $c=0.5$.}
	\label{fig:correctK}
\end{figure}
\begin{figure}[t!]
	\begin{center}
		\subfigure[$n=100$]{\includegraphics[width=0.7\linewidth]{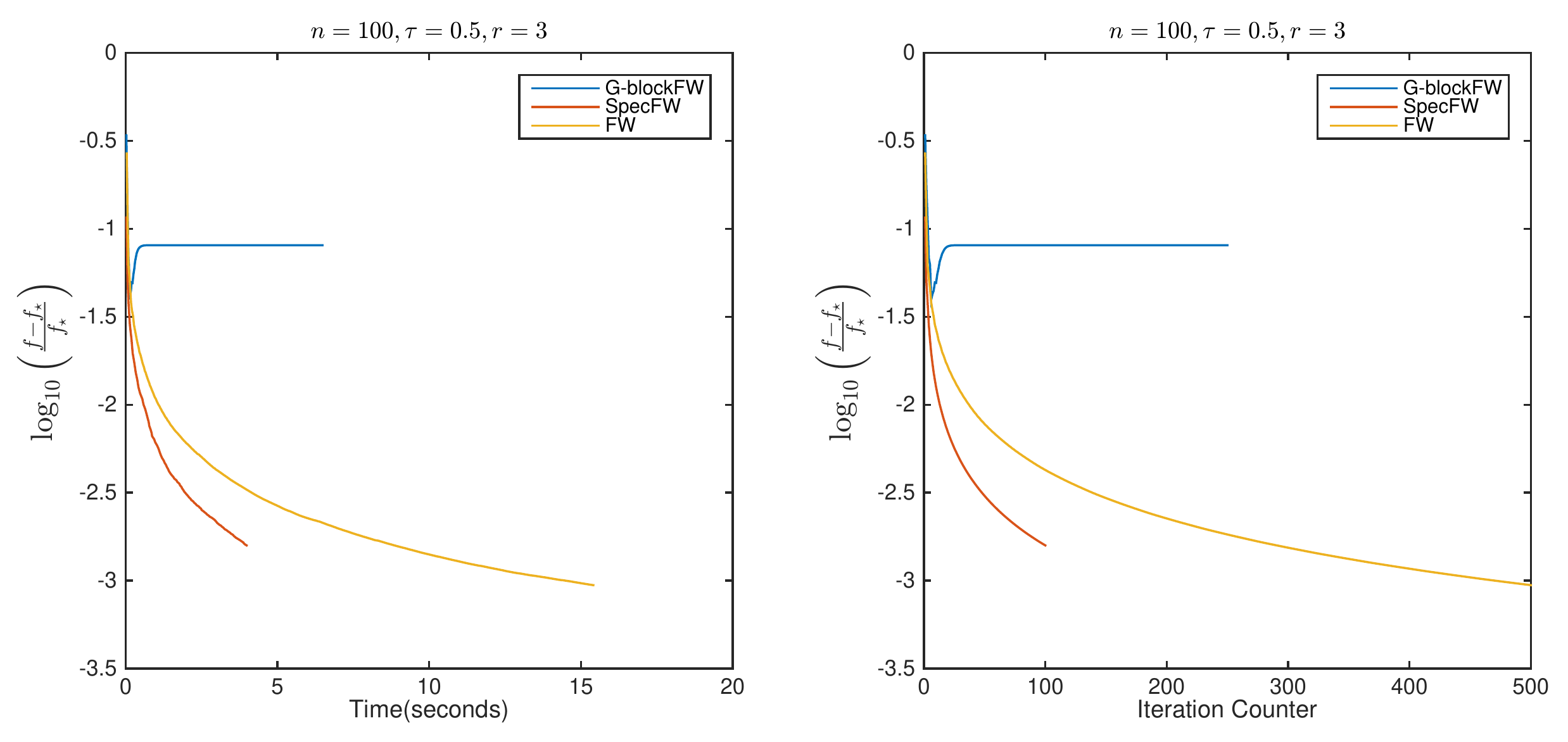}}
		\subfigure[$n=200$]{\includegraphics[width=0.7\linewidth]{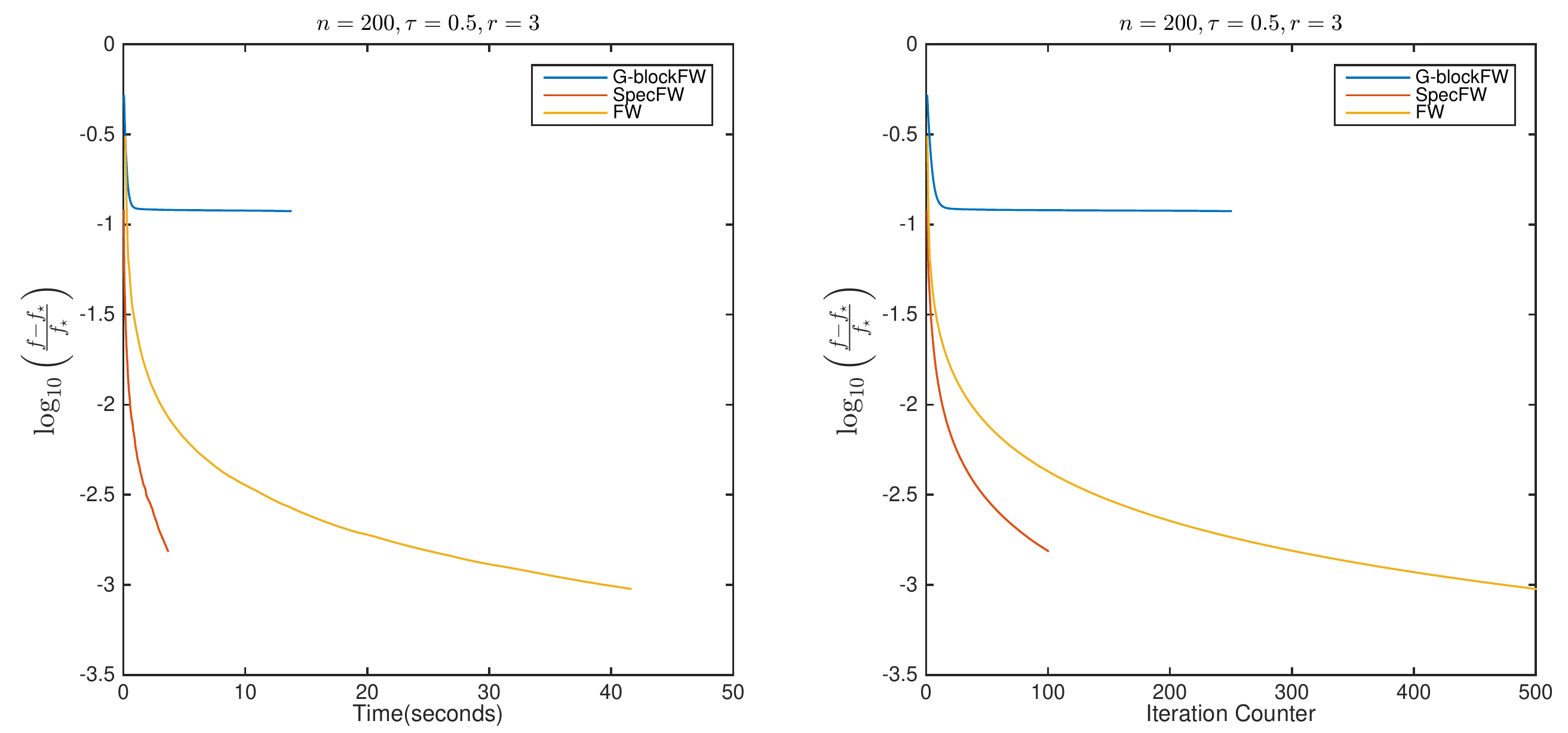} }
			\subfigure[$n=400$]{\includegraphics[width=0.7\linewidth]{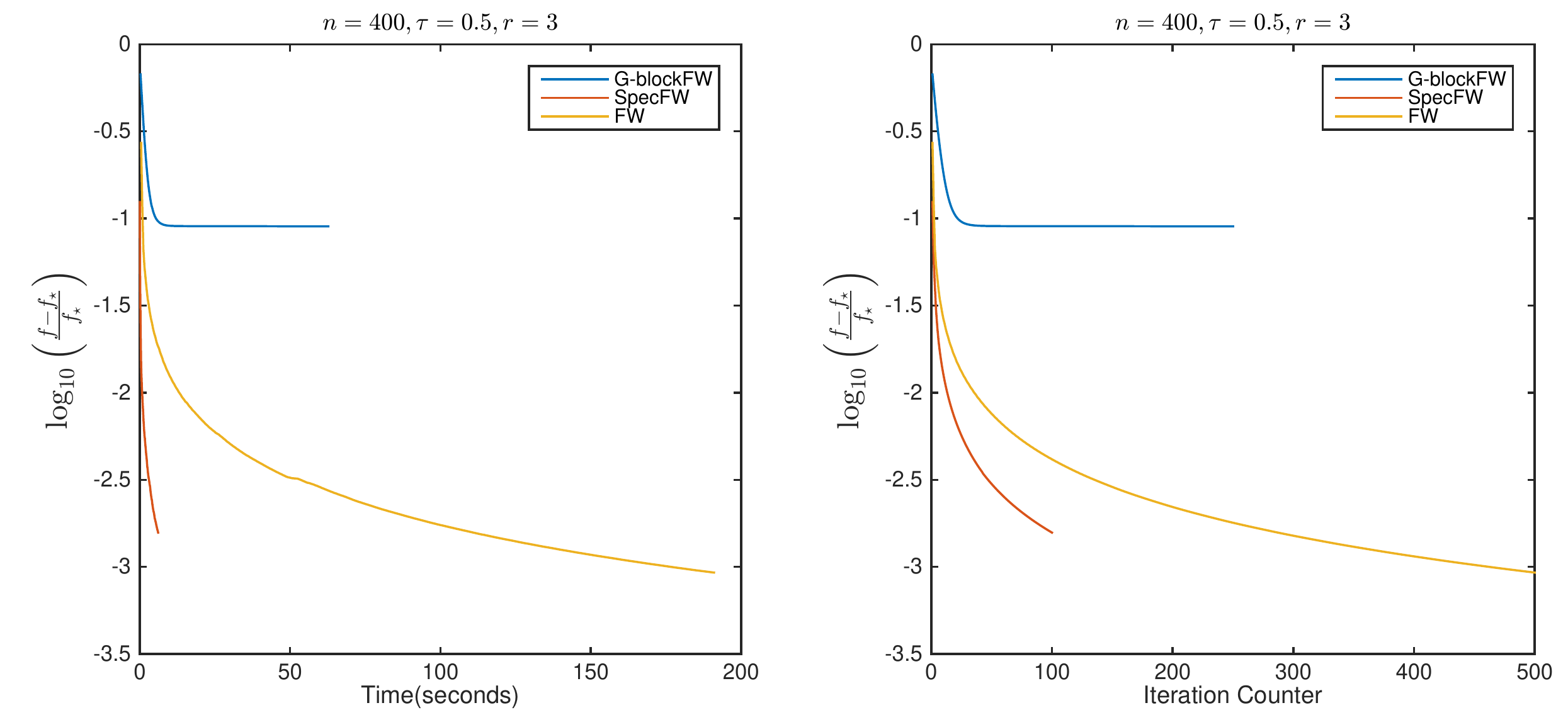} }
	\end{center}
	\caption{Comparison of algorithms under $\tau=\frac{1}{2}$, noise level $c=0.5$, and $k=2<\rsol$.}
	\label{fig:missK}
\end{figure}
\begin{figure}[t!]
	\begin{center}
		\subfigure[$n=100$]{\includegraphics[width=0.7\linewidth]{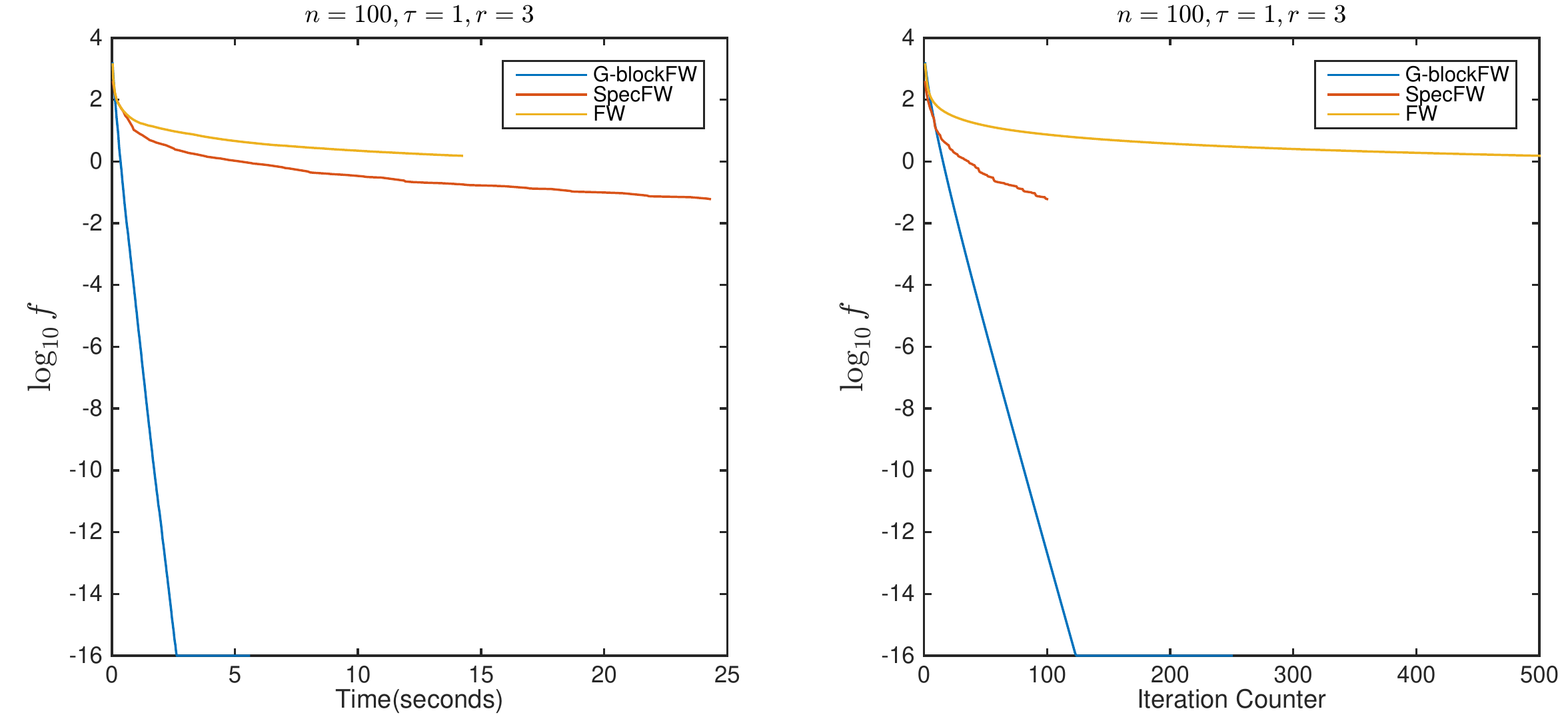}}
		\subfigure[$n=200$]{\includegraphics[width=0.7\linewidth]{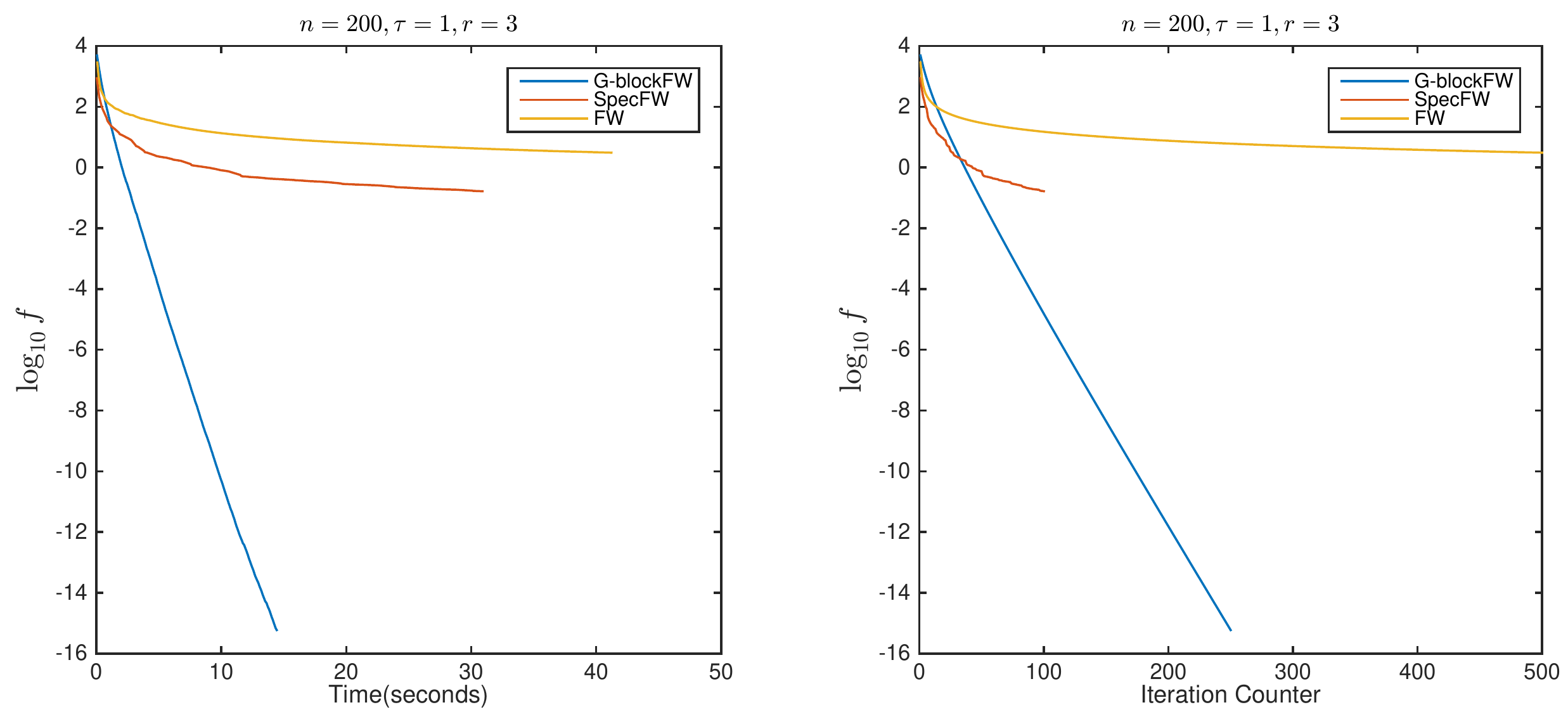} }
			\subfigure[$n=400$]{\includegraphics[width=0.7\linewidth]{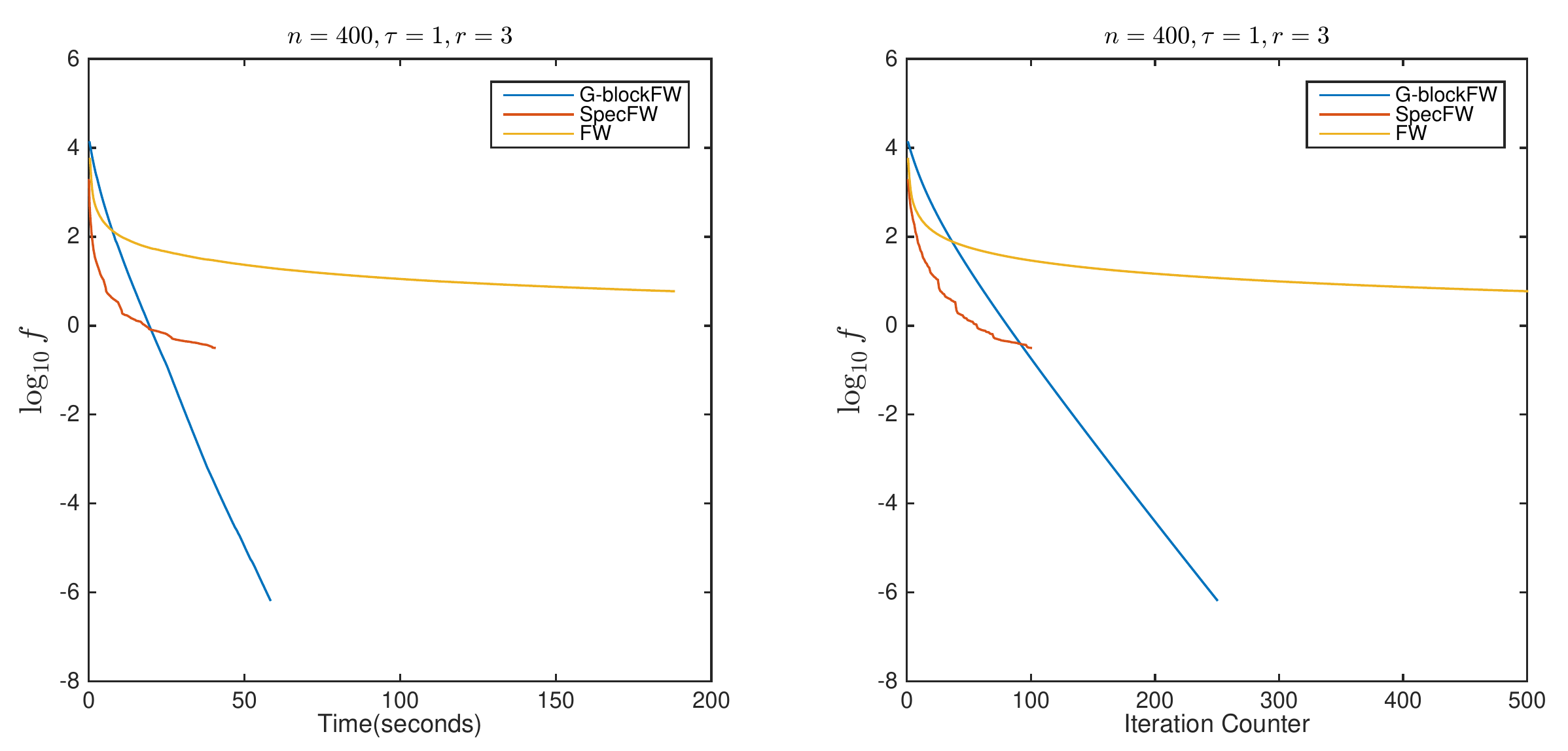} }
	\end{center}
	\caption{Comparison of algorithms under $\tau=1$, noise level $c=0$, and $k=4>\rsol$.}
	\label{fig:Nonoise}
\end{figure}



%% file: MainText.bbl
\begin{thebibliography}{26}
\providecommand{\natexlab}[1]{#1}
\providecommand{\url}[1]{\texttt{#1}}
\expandafter\ifx\csname urlstyle\endcsname\relax
  \providecommand{\doi}[1]{doi: #1}\else
  \providecommand{\doi}{doi: \begingroup \urlstyle{rm}\Url}\fi

\bibitem[Ahmed et~al.(2013)Ahmed, Recht, and Romberg]{ahmed2013blind}
Ahmed, A., Recht, B., and Romberg, J.
\newblock Blind deconvolution using convex programming.
\newblock \emph{IEEE Transactions on Information Theory}, 60\penalty0
  (3):\penalty0 1711--1732, 2013.

\bibitem[Alizadeh et~al.(1997)Alizadeh, Haeberly, and
  Overton]{alizadeh1997complementarity}
Alizadeh, F., Haeberly, J.-P.~A., and Overton, M.~L.
\newblock Complementarity and nondegeneracy in semidefinite programming.
\newblock \emph{Mathematical programming}, 77\penalty0 (1):\penalty0 111--128,
  1997.

\bibitem[Allen-Zhu et~al.(2017)Allen-Zhu, Hazan, Hu, and Li]{allen2017linear}
Allen-Zhu, Z., Hazan, E., Hu, W., and Li, Y.
\newblock Linear convergence of a frank-wolfe type algorithm over trace-norm
  balls.
\newblock In \emph{Advances in Neural Information Processing Systems}, pp.\
  6191--6200, 2017.

\bibitem[Bauschke et~al.(1999)Bauschke, Borwein, and Li]{bauschke1999strong}
Bauschke, H.~H., Borwein, J.~M., and Li, W.
\newblock Strong conical hull intersection property, bounded linear regularity,
  jameson's property (g), and error bounds in convex optimization.
\newblock \emph{Mathematical Programming}, 86\penalty0 (1):\penalty0 135--160,
  1999.

\bibitem[Cand{\`e}s \& Recht(2009)Cand{\`e}s and Recht]{candes2009exact}
Cand{\`e}s, E.~J. and Recht, B.
\newblock Exact matrix completion via convex optimization.
\newblock \emph{Foundations of Computational mathematics}, 9\penalty0
  (6):\penalty0 717, 2009.

\bibitem[Candes et~al.(2015)Candes, Eldar, Strohmer, and
  Voroninski]{candes2015phase}
Candes, E.~J., Eldar, Y.~C., Strohmer, T., and Voroninski, V.
\newblock Phase retrieval via matrix completion.
\newblock \emph{SIAM review}, 57\penalty0 (2):\penalty0 225--251, 2015.

\bibitem[Chen et~al.(2015)Chen, Chi, and Goldsmith]{chen2015exact}
Chen, Y., Chi, Y., and Goldsmith, A.~J.
\newblock Exact and stable covariance estimation from quadratic sampling via
  convex programming.
\newblock \emph{IEEE Transactions on Information Theory}, 61\penalty0
  (7):\penalty0 4034--4059, 2015.

\bibitem[Drusvyatskiy \& Lewis(2011)Drusvyatskiy and
  Lewis]{drusvyatskiy2011generic}
Drusvyatskiy, D. and Lewis, A.~S.
\newblock Generic nondegeneracy in convex optimization.
\newblock \emph{Proceedings of the American Mathematical Society}, pp.\
  2519--2527, 2011.

\bibitem[Drusvyatskiy et~al.(2016)Drusvyatskiy, Ioffe, and
  Lewis]{drusvyatskiy2016generic}
Drusvyatskiy, D., Ioffe, A.~D., and Lewis, A.~S.
\newblock Generic minimizing behavior in semialgebraic optimization.
\newblock \emph{SIAM Journal on Optimization}, 26\penalty0 (1):\penalty0
  513--534, 2016.

\bibitem[Frank \& Wolfe(1956)Frank and Wolfe]{frank1956algorithm}
Frank, M. and Wolfe, P.
\newblock An algorithm for quadratic programming.
\newblock \emph{Naval research logistics quarterly}, 3\penalty0 (1-2):\penalty0
  95--110, 1956.

\bibitem[Freund et~al.(2017)Freund, Grigas, and Mazumder]{fd2017extended}
Freund, R.~M., Grigas, P., and Mazumder, R.
\newblock An extended frank-wolfe method with ``in-face" directions, and its
  application to low-rank matrix completion.
\newblock \emph{SIAM Journal on optimization}, 27\penalty0 (1):\penalty0
  319--346, 2017.

\bibitem[Garber(2016)]{garber2016faster}
Garber, D.
\newblock Faster projection-free convex optimization over the spectrahedron.
\newblock In \emph{Advances in Neural Information Processing Systems}, pp.\
  874--882, 2016.

\bibitem[Garber(2019{\natexlab{a}})]{garber2019convergence}
Garber, D.
\newblock On the convergence of projected-gradient methods with low-rank
  projections for smooth convex minimization over trace-norm balls and related
  problems.
\newblock \emph{arXiv preprint arXiv:1902.01644}, 2019{\natexlab{a}}.

\bibitem[Garber(2019{\natexlab{b}})]{garber2019linear}
Garber, D.
\newblock Linear convergence of frank-wolfe for rank-one matrix recovery
  without strong convexity.
\newblock \emph{arXiv preprint arXiv:1912.01467}, 2019{\natexlab{b}}.

\bibitem[Goldstein et~al.(2014)Goldstein, Studer, and
  Baraniuk]{GoldsteinStuderBaraniuk:2014}
Goldstein, T., Studer, C., and Baraniuk, R.
\newblock A field guide to forward-backward splitting with a {FASTA}
  implementation.
\newblock \emph{arXiv eprint}, abs/1411.3406, 2014.
\newblock URL \url{http://arxiv.org/abs/1411.3406}.

\bibitem[Goldstein et~al.(2015)Goldstein, Studer, and Baraniuk]{FASTA:2014}
Goldstein, T., Studer, C., and Baraniuk, R.
\newblock {FASTA}: A generalized implementation of forward-backward splitting,
  January 2015.
\newblock http://arxiv.org/abs/1501.04979.

\bibitem[Jaggi(2013)]{jaggi2013revisiting}
Jaggi, M.
\newblock Revisiting frank-wolfe: Projection-free sparse convex optimization.
\newblock In \emph{Proceedings of the 30th international conference on machine
  learning}, pp.\  427--435, 2013.

\bibitem[Jaggi \& Sulovsk{\`y}(2010)Jaggi and Sulovsk{\`y}]{jaggi2010simple}
Jaggi, M. and Sulovsk{\`y}, M.
\newblock A simple algorithm for nuclear norm regularized problems.
\newblock In \emph{Proceedings of the 27th International Conference on
  International Conference on Machine Learning}, pp.\  471--478, 2010.

\bibitem[Karimi et~al.(2016)Karimi, Nutini, and Schmidt]{karimi2016linear}
Karimi, H., Nutini, J., and Schmidt, M.
\newblock Linear convergence of gradient and proximal-gradient methods under
  the polyak-{\l}ojasiewicz condition.
\newblock In \emph{Joint European Conference on Machine Learning and Knowledge
  Discovery in Databases}, pp.\  795--811. Springer, 2016.

\bibitem[Kuczy{\'n}ski \& Wo{\'z}niakowski(1992)Kuczy{\'n}ski and
  Wo{\'z}niakowski]{kuczynski1992estimating}
Kuczy{\'n}ski, J. and Wo{\'z}niakowski, H.
\newblock Estimating the largest eigenvalue by the power and lanczos algorithms
  with a random start.
\newblock \emph{SIAM journal on matrix analysis and applications}, 13\penalty0
  (4):\penalty0 1094--1122, 1992.

\bibitem[Necoara et~al.(2019)Necoara, Nesterov, and Glineur]{necoara2019linear}
Necoara, I., Nesterov, Y., and Glineur, F.
\newblock Linear convergence of first order methods for non-strongly convex
  optimization.
\newblock \emph{Mathematical Programming}, 175\penalty0 (1-2):\penalty0
  69--107, 2019.

\bibitem[Nesterov(2013)]{nesterov2013introductory}
Nesterov, Y.
\newblock \emph{Introductory lectures on convex optimization: A basic course},
  volume~87.
\newblock Springer Science \& Business Media, 2013.

\bibitem[Recht et~al.(2010)Recht, Fazel, and Parrilo]{recht2010guaranteed}
Recht, B., Fazel, M., and Parrilo, P.~A.
\newblock Guaranteed minimum-rank solutions of linear matrix equations via
  nuclear norm minimization.
\newblock \emph{SIAM review}, 52\penalty0 (3):\penalty0 471--501, 2010.

\bibitem[Trefethen \& Bau~III(1997)Trefethen and
  Bau~III]{trefethen1997numerical}
Trefethen, L.~N. and Bau~III, D.
\newblock \emph{Numerical linear algebra}, volume~50.
\newblock Siam, 1997.

\bibitem[Tropp et~al.(2017)Tropp, Yurtsever, Udell, and
  Cevher]{tropp2017practical}
Tropp, J.~A., Yurtsever, A., Udell, M., and Cevher, V.
\newblock Practical sketching algorithms for low-rank matrix approximation.
\newblock \emph{SIAM Journal on Matrix Analysis and Applications}, 38\penalty0
  (4):\penalty0 1454--1485, 2017.

\bibitem[Yurtsever et~al.(2017)Yurtsever, Udell, Tropp, and
  Cevher]{yurtsever2017sketchy}
Yurtsever, A., Udell, M., Tropp, J., and Cevher, V.
\newblock Sketchy decisions: Convex low-rank matrix optimization with optimal
  storage.
\newblock In \emph{Artificial Intelligence and Statistics}, pp.\  1188--1196,
  2017.

\end{thebibliography}
